\documentclass[12pt]{amsart}
\usepackage[margin=1.0in]{geometry} 
\usepackage{esint,amsthm,amsmath,amssymb,amsfonts,graphicx,color,comment,enumerate,psfrag,caption,bbm,bm,hyperref,enumitem,mathrsfs}
\usepackage[hyperpageref]{backref}
\usepackage{mathtools}

\usepackage{tikz,pgf}
\usetikzlibrary{calc}
\usetikzlibrary{patterns}

\usepackage{cleveref}
\allowdisplaybreaks

\DeclareMathOperator*{\esssup}{ess\,sup}
\DeclareMathOperator\supp{supp}

\newtheorem{lemma}{Lemma}[section]
\newtheorem{remark}{Remark}[section]
\numberwithin{equation}{section}
\newtheorem{theorem}{Theorem}[section]
\newtheorem{proposition}[theorem]{Proposition}

\title{Bochner-Riesz commutators for Grushin Operators}

\author[Md N. Molla, J. Singh]
{Md Nurul Molla \and Joydwip Singh} 

\address[Md N. Molla]{Department of Mathematics, Indian Institute of Science Education and Research Kolkata, Mohanpur--741246, West Bengal, India.}
\email{nurul.pdf@iiserkol.ac.in}

\address[J. Singh]{Department of Mathematics and Statistics, Indian Institute of Science Education and Research Kolkata, Mohanpur--741246, West Bengal, India.}
\email{js20rs078@iiserkol.ac.in}



\begin{document}

\begin{abstract}
In this paper, we study the boundedness of Bochner-Riesz commutator $$[b, S^{\alpha}(\mathcal{L})](f) = b S^{\alpha}(\mathcal{L})(f) - S^{\alpha}(\mathcal{L})(bf)$$ of a $BMO^{\varrho}(\mathbb{R}^d)$ function $b$ and the Bochner-Riesz operator $S^{\alpha}(\mathcal{L})$ associated to the Grushin operator $\mathcal{L}$ on $\mathbb{R}^d$ with $d:= d_1 +d_2$. We prove that for $1\leq p \leq \min \{2d_1/(d_1 +2), 2(d_2 +1)/(d_2+3)\}$ and  $\alpha > d(1/p - 1/2) - 1/2$, if $b \in BMO^{\varrho}(\mathbb{R}^d)$, then $[b, S^{\alpha}(\mathcal{L})]$ is bounded on $L^q(\mathbb{R}^d)$ whenever $p < q < p'$. Moreover, if $b \in CMO^{\varrho}(\mathbb{R}^d)$, then we show that $[b, S^{\alpha}(\mathcal{L})]$ is a compact operator on $L^q(\mathbb{R}^d)$ in the same range.
\end{abstract}

\maketitle

\section{Introduction}
The Bochner-Riesz operator $S^{\alpha}$ of order $\alpha \geq 0$ in $\mathbb{R}^n$ is defined by
\begin{align*}
    S^{\alpha}(f)(x) = \int_{\mathbb{R}^n} \left(1-|\xi|^2\right)_{+}^{\alpha} \widehat{f}(\xi) e^{2 \pi i x \cdot \xi} \ d\xi,
\end{align*}
where $(r)_{+} = \max\{r, 0\}$ for $r \in \mathbb{R}$, $f \in \mathcal{S}(\mathbb{R}^n)$, the Schwartz class functions in $\mathbb{R}^n$, and $\widehat{f}$ is its Fourier transform. Related to the summability of Fourier series, the $L^p$- boundedness of Bochner-Riesz operator is one of the most fundamental problem in harmonic analysis, and still remains open for $n \geq 3$ and $1 \leq p \leq \infty$. For $1\leq p \leq \infty$ and $p\neq 2$, it is conjectured that $S^{\alpha}$ is bounded on $L^p(\mathbb{R}^n)$ if and only if $\alpha> \alpha(p) = \max\left\{n\left|\frac{1}{p}-\frac{1}{2}\right|-\frac{1}{2}, 0 \right\}$. When $\alpha=0$, $S^{\alpha}$ is the ball multiplier operator and Fefferman showed that it is unbounded on $L^p(\mathbb{R}^n)$ for $n \geq 2$, except $p=2$. For $n=2$, the conjecture was shown to be true by Carleson and Sj\"olin. However, for $n \geq 3$ the conjecture is verified on a restricted range and remains open in general.

The Bochner-Riesz commutator $[b, S^{\alpha}]$ of a $BMO(\mathbb{R}^n)$ function $b$ and $S^{\alpha}$ is defined by
\begin{align*}
    [b, S^{\alpha}](f) = b S^{\alpha}(f) - S^{\alpha}(bf).
\end{align*}
The following results are well known in the literature regarding the $L^p$-boundedness of $[b, S^{\alpha}]$. If $\alpha \geq \frac{n-1}{2}$, then $[b, S^{\alpha}]$ is bounded on $L^p(\mathbb{R}^n)$ for $1<p<\infty$ (see \cite{Lu_Yan_Bochner_Riesz_Book_2013}). On $\mathbb{R}^n$ for $n\geq 2$, \cite{Hu_Lu_Bochner_Riesz_Commutator_1996} and \cite{Lu_Xia_Bochner_Riesz_Commutator_2007} proved that if $\frac{n-1}{2n+2}<\alpha<\frac{n-1}{2}$ then $[b, S^{\alpha}]$ is bounded on $L^p(\mathbb{R}^n)$ if and only if $\frac{2n}{n+1+2\alpha}<p< \frac{2n}{n-1-2\alpha}$. Uchiyama laid the foundation of the study of the compactness of commutators for some important class of operators in harmonic analysis. In his pioneering work \cite{Uchiyama_compactness_commutator_1978}, he characterized all locally integrable  functions $b$ on $\mathbb{R}^n$ such that for any $1 < p < \infty$ and  any Calder\'on-Zygmund operator $T$ with smooth kernel, $[b, T]$ is compact on $L^p(\mathbb{R}^n)$. The function $b$ turned out in the space $ CMO(\mathbb{R}^n)$ which is the closure in $BMO(\mathbb{R}^n)$ of the space of $C_c^{\infty}(\mathbb{R}^n)$ functions. Since then, a lot of attention has been paid to the compactness of many different operators in harmonic analysis, and  many remarkable results are known. Interested reader are refer to \cite{Duong_Li_Mao_Wu_yang_Compactness_Riesz_transform_2018}, \cite{Chen_Duong_Li_Wu_Compactness_Riesz_commutator_2019}, \cite{Tao_Yang_Yuan_Zhang_Compactness_commutator_2023} and the references therein for glimpse of compactness of some important operators. The compactness of Bochner-Riesz commutator in $\mathbb{R}^n$ has been studied in \cite{Bu_Chen_Hu_Compactness_commutator_2017}.

Beyond the Euclidean setup, the Bochner-Riesz operators and its commutator also have been studied by many authors. In case of Hermite operator see \cite{Chen_Lin_Yan_Commutator_Bochner_Riesz_Hermite_2023}. On metric measure space $X$ with certain assumptions, for a non-negative self-adjoint operator $L$ on $L^2(X)$, the boundedness of Bochner-Riesz operator $S^{\alpha}(L)$ has been studied by \cite{Chen_Ouhabaz_Sikora_Yan_Restriction_Estimate_Bochner-Riesz_2016}. In \cite{Bui_Commutators_Spectral_Multiplier_2012} the boundedness of $[b, S^{\alpha}(L)]$ was proved for $\alpha> \frac{n-1}{2}$ on $L^p(X)$ for all $1<p<\infty$. On the other hand, when $0< \alpha \leq \frac{n-1}{2}$, Chen et. al. proved the following result.
\begin{theorem}[\cite{Chen_Tian_Ward_Commutator_Bochner_Riesz_Elliptic_2021}, Theorem 1.1.]
    Assume that on $X$ ball volume is polynomial. Suppose that $L$ satisfies the finite speed propagation property and restriction type condition for some $1\leq p<2$. Let $b \in BMO(X)$. Then for all $\alpha> n(1/p-1/2)-1/2$, the $[b, S^{\alpha}(L)]$ is bounded on $L^q(X)$ for all $p<q<p'$.
\end{theorem} 


In this paper, we are concern about the Grushin operator on $\mathbb{R}^d := \mathbb{R}^{d_1} \times \mathbb{R}^{d_2}$, $d_1, d_2 \geq 1$ with $d=d_1 +d_2$. The Grushin operator $\mathcal{L}$ is defined by
\begin{align*}
    \mathcal{L} &= -\Delta_{x'} - |x'|^2 \Delta_{x''} ,
\end{align*}
where $x=(x',x'')\in \mathbb{R}^{d_1} \times \mathbb{R}^{d_2}$, while $\Delta_{x'}$, $\Delta_{x''}$ are the Laplacian on $\mathbb{R}^{d_1}$, $\mathbb{R}^{d_2}$ respectively and $|x'|$ is the Euclidean norm of $x'$. The operator $\mathcal{L}$
is positive and essentially self-adjoint on $L^2(\mathbb{R}^d)$ but not elliptic on the plane $x' = 0$. More details can be found in \cite{Martini_Sikora_Grushin_Weighted_Plancherel_2012}, \cite{Martini_Muller_Multiplier_Grushin_2014}, \cite{Robinson_Sikora_Analysis_Grusin_Operator_2008} and references therein.

The spectral decomposition of Grushin operator is well known (see \cite{Bagchi_Garg_Grushin_2024}). For $f \in \mathcal{S}(\mathbb{R}^d)$ let $f^{\lambda}(x') = \int_{\mathbb{R}^{d_2}} f(x,x'') e^{-i \lambda \cdot x''} \ dx''$, denote the Fourier transform in the second variable $x''$. Then Bochner-Riesz means for Grushin operator is defined by
\begin{align*}
    S^{\alpha}(\mathcal{L}) f(x) &= \frac{1}{(2\pi)^{d_2}} \int_{\mathbb{R}^{d_2}} \sum_{k=0}^{\infty} \left(1-(2k+d_1)|\lambda| \right)_{+}^{\alpha} P_k^{\lambda}f^{\lambda}(x') e^{i \lambda \cdot x''}\ d\lambda ,
\end{align*}
where $P_k^{\lambda}$ denotes the projection of scaled Hermite operator onto the eigenspace associated to the eigenvalue $(2k+d_1)|\lambda|$, (see \cite{Jotsaroop_Sanjay_Thangavelu_Grushin_2013}).

As we have seen above that the Grushin operator is not elliptic, but it is sub-elliptic. For sub-elliptic operators, there are two types of dimensions of the underlying space: one is called the \emph{homogeneous dimension} and the other is the \emph{topological dimension}. And in general homogeneous dimension is strictly bigger than the topological dimension. In recent times, proving sharp spectral multipliers results  in terms of smoothness properties of the multiplier $F$ featuring topological dimension instead of homogeneous dimension of the underline space has been the subject of great interest. The genesis of this problem can be traced down to the seminal works of M\"uller and Stein \cite{Muller_Stein_Spectral_Multiplier-Heisenberg_1994} and  Hebisch \cite{Hebisch_Spectral_Multiplier_Heisenberg_1993} in the context of Heisenberg group. 

Like Heisenberg group, spectral multipliers and boundedness of Bochner-Riesz operators for Grushin operator have been studied recently by many authors. In \cite{Martini_Sikora_Grushin_Weighted_Plancherel_2012}, Martini and Sikora proved the H\"ormander-Mikhlin multiplier theorem for the Grushin operator and from that they deduced that if $\alpha> \max\{d_1+d_2, 2d_2\}/2-1/2$, then the Bochner-Riesz mean $S^{\alpha}(\mathcal{L})$ is bounded on $L^p(\mathbb{R}^d)$ for all $1\leq p \leq \infty$. After that Martini and M\"uller, \cite{Martini_Muller_Multiplier_Grushin_2014} improved the condition to $\alpha>\frac{d-1}{2}$, where $d=d_1 + d_2$ is the topological dimension, which is strictly smaller than the homogeneous dimension $Q=d_1+2d_2$ of $\mathbb{R}^d$. For  the case, $0< \alpha \leq \frac{d-1}{2}$ the scenario is quiet different. One cannot expect the result to hold for all $p \in [1, \infty]$. In this situation,  Chen and Ouhabaz, in \cite{Chen_Ouhabaz_Bochner-Riesz_Grushin_2016} proved the $p$-specific boundedness result for Grushin operator with smoothness indices $\alpha> \max\{d_1+d_2, 2d_2\}|1/p-1/2|-1/2 $ and recently which was improved by Niedorf in \cite{Niedorf_Grushin_Multiplier_2022}. He proved that for $1 \leq  p \leq \min{\{2d_1/(d_1 + 2), (2d_2 + 2)/(d_2 + 3)\}}$ and $\alpha >d|1/p-1/2|-1/2$, the Bochner–Riesz mean $S^{\alpha}(\mathcal{L})$ is bounded on $L^p(\mathbb{R}^d )$.

In this paper we study the $L^q$-boundedness and compactness of the Bochner-Riesz commutator associated to the Grushin operator $\mathcal{L}$. For $b\in BMO^{\varrho}(\mathbb{R}^d)$ (see (\ref{Definition of BMO})) and $f \in C_c^{\infty}(\mathbb{R}^d)$ the Bochner-Riesz commutator $[b, S^{\alpha}(\mathcal{L})]$  of Bochner-Riesz operators $S^{\alpha}(\mathcal{L})$ is defined by
\begin{align*}
    [b, S^{\alpha}(\mathcal{L})]f = b S^{\alpha}(\mathcal{L})f - S^{\alpha}(\mathcal{L})(bf).
\end{align*}
Note that if $f \in C_c^{\infty}(\mathbb{R}^d)$ and $b\in BMO^{\varrho}(\mathbb{R}^d)$, then $bf \in L^p(\mathbb{R}^d)$ for all $p<\infty$. The first main result of this paper is the following.
\begin{theorem}\label{Theorem: Bochner-Riesz Commutator}
    Let $1\leq p \leq \min\{2d_1/(d_1+2), 2(d_2+1)/(d_2+3) \} $ and $\alpha>d(1/p-1/2)-1/2$. Then for $b \in BMO^{\varrho}(\mathbb{R}^d)$ we have
    \begin{align*}
        \|[b, S^{\alpha}(\mathcal{L})]f\|_{L^q(\mathbb{R}^d)} &\leq C \|b\|_{BMO^{\varrho}(\mathbb{R}^d)} \|f\|_{L^q(\mathbb{R}^d)},
    \end{align*}
    for all $p<q<p'$.
\end{theorem}

Instead of the above result, we shall prove the following result which is slightly more general in nature. Let us use the standard notation $W^2_{\beta}(\mathbb{R})$ for the Sobolev space with norm $\|f\|_{W^2_{\beta}(\mathbb{R})} = \|(I-d^2/dx^2)^{\beta/2}f \|_{L^2(\mathbb{R})}$.
\begin{theorem}
\label{Theorem: Multiplier for Commutator}
    Let $1\leq p \leq \min\{2d_1/(d_1+2), 2(d_2+1)/(d_2+3) \}$ and $\beta>d(1/p-1/2)$. Suppose $b \in BMO^{\varrho}(\mathbb{R}^d)$. Then for any even Borel function $F$, with $\supp{F} \subseteq [-1,1]$ and $F \in W^2_{\beta}(\mathbb{R})$ we have
    \begin{align}\label{main inequality}
        \|[b, F(t \sqrt{\mathcal{L}})]f\|_{L^q(\mathbb{R}^d)} &\leq C \|b\|_{BMO^{\varrho}(\mathbb{R}^d)} \|F\|_{W^2_{\beta}(\mathbb{R})} \|f\|_{L^q(\mathbb{R}^d)},
    \end{align}
    for all $p<q<p'$ and uniformly in $t \in (0, \infty)$.
\end{theorem}

Define $CMO^{\varrho}(\mathbb{R}^d) := \overline{C_c^{\infty}(\mathbb{R}^d)}^{BMO^{\varrho}(\mathbb{R}^d)}$, that is the closer of compactly supported smooth function in $BMO^{\varrho}(\mathbb{R}^d)$. Regarding the compactness of $[b, S^{\alpha}(\mathcal{L})]$ we have the following  second main result of this paper.
\begin{theorem}
\label{Theorem: Compactness of Bochner-Riesz commutator}
    Let $1\leq p \leq \min\{2d_1/(d_1+2), 2(d_2+1)/(d_2+3) \}$ and $\alpha>d(1/p-1/2)-1/2$. If $b \in CMO^{\varrho}(\mathbb{R}^d)$, then the Bochner-Riesz commutator $[b, S^{\alpha}(\mathcal{L})]$ is a compact operator on $L^q(\mathbb{R}^d)$ for all $p<q<p'$.
\end{theorem}

\begin{remark}
    Notice that in both our results, the smoothness parameter $\alpha$ is expressed in terms of the topological dimension $d$, of the underlying space $\mathbb{R}^d$. It is also important to remark that the threshold $\alpha>d(1/p-1/2)-1/2$ is  coincides with that of the Bochner-Riesz conjecture for the Grushin operator, which is optimal and can not be decreased further (\cite{Niedorf_Grushin_Multiplier_2022}). The main novelty of our work is to reduce the dimension from $Q:=d_1+2d_2$, the homogeneous dimension of the space $\mathbb{R}^d$ to $d:=d_1+d_2$, the topological dimension of  $\mathbb{R}^d$. There has been extensive research on dimension reduction problem for multiplier \cite{Martini_Sikora_Grushin_Weighted_Plancherel_2012}, \cite{Martini_Muller_Multiplier_Grushin_2014}, \cite{Muller_Stein_Spectral_Multiplier-Heisenberg_1994}, \cite{Marini_Multiplier_polynomial_Growth_2012}, yet to the best of our knowledge, for commutators we are not aware of this kind of results in the literature.
\end{remark}

\subsection{Sketch of proof of the Theorems} Concerning $L^p$-boundedness of Bochner-Riesz commutator on the classical Euclidean space $\mathbb{R}^n$ for the Laplacian $\Delta= \sum_{i = 1} ^n \partial_{x_i} ^2$, or more generally in the context of metric measure space $X$ for self-adjoint elliptic operator $L$ \cite{Chen_Tian_Ward_Commutator_Bochner_Riesz_Elliptic_2021}, one of the crucial assumption was made that the ball volume is polynomial type. The first main difficulty in our settings is that, the volume is no longer of polynomial type but rather depends on the the radius as well as the centre of the ball. This requires us to distinguish two different cases and further careful analysis. Another essential aspect of their proof is the employment of the celebrated Stein-Tomas type restriction type condition: 
\begin{align*}
    \|dE_{\sqrt{L}} (\lambda)  \|_{p \rightarrow p'} \leq C \lambda^{n(1/p - 1/p') -1}, \quad \lambda >0, 
\end{align*}
for some $1 \leq p <2$. Even if such estimates are available, this only leads us to one of our cases with the smoothness parameter $\alpha > Q(1/p - 1/2)$ rather than $d(1/p - 1/2)$. To overcome this obstacle, one needs different idea.

One way to overcome this problem is to use weight $|x'|^{\gamma}$ in the first layer $x'$ as in \cite{Martini_Sikora_Grushin_Weighted_Plancherel_2012} with weighted restriction estimate (Theorem 3.4 of \cite{Chen_Ouhabaz_Bochner-Riesz_Grushin_2016}) and modify the arguments in \cite{Chen_Tian_Ward_Commutator_Bochner_Riesz_Elliptic_2021}. However, this approach will only give us smoothness indices $\alpha> \max\{d_1 +d_2, 2d_2\}/2-1/2$. Note that this result yields the smoothness indices $\alpha$ in terms of the topological dimension $d$ only if $d_1 \geq d_2$. Therefore as in \cite{Martini_Muller_Multiplier_Grushin_2014}, one can think of use weights $|x''|^{\gamma} (\gamma>0)$ in the second layer $x''$. But as pointed out in \cite{Niedorf_Grushin_Multiplier_2022}, the weights $|x''|^{\gamma}$ are difficult to handle due to the sub-elliptic estimate of Hebisch \cite{Hebisch_Spectral_Multiplier_Heisenberg_1993}, is not applicable for these weights. To tackle down the problem irrespective of all $d_1, d_2$, we use a weaker version of restriction type estimates Theorem \ref{Theorem: Restriction Estimate}, which is due to \cite{Niedorf_Grushin_Multiplier_2022},  where the operator $F(L)$ is further truncated dyadically  along the spectrum of $T$. Nevertheless, such an approach is limited to dealing with the error part of this additional truncation. In order to address the main part, the sub-Riemannian geometry of the underlying manifold is taken into account. This helps us to further split the main part into two summand. The negligible part of the summand is tackled through the weighted Plancherel estimates (Proposition \ref{Theorem: Weighted Plancherel in second variable}), whereas the significant part of the summand is more involved.

In Euclidean space, the compactness of the commutator of Bochner-Riesz multiplier is obtained by using  Fourier transform estimates, approximation to the identity, and some refined estimates obtained by C. Fefferman in \cite{Fefferman_Spherical_multiplier_1973}. To prove the compactness of Bochner-Riesz commutator $[b, F(\sqrt{\mathcal{L}})]$, we use Kolmogorov-Riesz compactness theorem established in \cite{Olsen_Holden_Kolmogorov_Riesz_Compactness_2010} (see \Cref{Theorem: Characterizations of compactness}). Unlike the Euclidean setting here we first decompose the commutator $[b, F(\sqrt{\mathcal{L}})]$ in the Fourier transform side of $F$ into sum of $[b, F^{(l)}(\sqrt{\mathcal{L}})]$ and this will reduce the problem to the compactness of the corresponding commutator $[b, F^{(l)}(\sqrt{\mathcal{L}})]$ for each $l \geq 0$, where we have used the crucial fact that norm limit of compact operators are again compact. This is an important observation of our paper. Now we need good control for each of these pieces in terms of $L^q$-norm. This amounts to establishing suitable $L^q$-estimates for $[b, F^{(l)}(\sqrt{\mathcal{L}})]$ (see equation (3.2)).

Throughout the article we  use standard notation. We use letter $C$ to indicate a  positive constant independent of the main parameters, but may vary from line to line. While writing estimates, we shall use the notation $f \lesssim g$ to indicate $f \leq Cg$ for some
$C > 0$, and whenever $f \lesssim g\lesssim f$ , we shall write $f \sim g$. Also, we write $f \lesssim_{\epsilon} g$ when the implicit constant $C$ may depend on a parameter like $\epsilon$. For a Lebesgue measurable subset $E$ of $\mathbb{R}^d$, we denote by $\chi_E$ the characteristic function of the set $E$. Also, for any ball $ B := B(x, r)$ with centered at $x$ and radius $r$, the notation $B(x, \kappa r)$ stands for the concentric dilation of $B$ by $\kappa >0$. Moreover, for a measurable function $f$, we denote the average of $f$ over $B$ as  $f_{B} = \frac{1}{|B|} \int_{B} f(x) \, dx$ and set $P_Bf(x) = \chi_B(x) f(x)$. For any function $G$ on $\mathbb{R}$, define $\delta_R G(\eta) = G(R \eta) $. Let $\mathcal{M}$ denote the Hardy-Littlewood maximal operator defined on $\mathbb{R}^d$ relative to the distance $\varrho$. 

The article is organized as follows.  In the next section, we collect some preliminary results. First, we recall some useful properties related to the Grushin operator $\mathcal{L}$. Next, we discuss about the truncated restriction type estimates, weighted Plancherel estimates and pointwise kernel estimate associated with $\mathcal{L}$. Moreover, we also recall the notion of the function space $BMO$ and some important properties of this space. Towards the end of this section we recall Kolmogorov-Riesz compactness theorem. Finally in Section 3 and 4, we give proof of our main results \Cref{Theorem: Bochner-Riesz Commutator}, \Cref{Theorem: Multiplier for Commutator} and \Cref{Theorem: Compactness of Bochner-Riesz commutator} of this paper.

\section{Preliminaries}\label{section preli}
In order to prove our results,  we will require some preliminary results. It will be convenient to gather them here first. On $\mathbb{R}^d := \mathbb{R}^{d_1} \times \mathbb{R}^{d_2}$, $d_1, d_2 \geq 1$ with $d=d_1 +d_2$ the Grushin operator $\mathcal{L}$ is defined by
\begin{align*}
    \mathcal{L} = - \sum_{j=1}^{d_1} \partial_{x'_j}^2 - \left(\sum_{j=1}^{d_1} |x'_j|^2 \right) \sum_{k=1}^{d_2} \partial_{x''_k}^2 .
\end{align*} Due to a celebrated theorem of H\"ormander \cite{Hormander_Hypoelliptic_Differential_Operator_1967}, the Grushin operator $\mathcal{L}$ is hypoelliptic second order differential operator with smooth coefficients. For these type of operators, there are notions of control distance available in the literature, for more details see \cite{Robinson_Sikora_Analysis_Grusin_Operator_2008}, \cite{Martini_Sikora_Grushin_Weighted_Plancherel_2012}. Let $\Tilde{\varrho}$ denote the control distance associated with the Grushin operator $\mathcal{L}$. Sometimes it is convenient to work with some explicit expression of the distance $\Tilde{\varrho}$, which is given by the following formula. For all $x,y \in \mathbb{R}^d$,
    \begin{align*}
    \Tilde{\varrho}(x,y) \sim \varrho(x,y) = |x'-y'| + \left\{ \begin{array}{ll}
       \frac{|x''-y''|}{|x'|+|y'|}  & \ \text{if}\ |x''-y''|^{1/2} \leq |x'|+ |y'| \\
      |x''-y''|^{1/2}   & \ \text{if}\  |x''-y''|^{1/2} \geq |x'|+ |y'| .
      \end{array} \right.
\end{align*}

On $\mathbb{R}^d$, we have a family of non-isotropic dilation $\{\delta_t\}_{t>0}$ defined by $\delta_t(x', x'')= (t x', t^2 x'') $. Then one can easily check that for all $t>0$,
\begin{align}
\label{Effect of dilation}
    \mathcal{L}(f \circ \delta_t) = t^2 (\mathcal{L}f) \circ \delta_{t} \quad \text{and} \quad \varrho(\delta_t x, \delta_t y) = t \varrho(x,y) .
\end{align}
Let $ B^{\varrho}(x,r) := \{y \in \mathbb{R}^d : \varrho(x,y) \leq r \}$ denote the $\varrho$-ball with center at $x \in \mathbb{R}^d$ and radius $r\geq 0$. Moreover, if $|B^{\varrho}(x, r)|$ denotes the Lebesgue measure of the $\varrho$-ball $B^{\varrho}(x,r)$, then
\begin{align}
   |B^{\varrho}(x, r)| \sim r^{d_1+ d_2}\, \max\{{r, |x'|}\}^{d_2}. 
\end{align}
In particular for all $\kappa \geq 0$ we have, 
\begin{align*}
    |B^{\varrho}(x, \kappa r)| \leq C (1 + \kappa)^{Q}\, |B^{\varrho}(x, r)|.
\end{align*}
Therefore $\mathbb{R}^d$ with the distance $\varrho$ and the Lebesgue measure $|\cdot|$ become a doubling metric measure space with ``homogeneous dimension" $Q:= d_1 + 2d_2$. We call $d= d_1 + d_2$ to be the ``topological dimension" of $\mathbb{R}^d$. Here onwards we omit the superscript $\varrho$ from $B^{\varrho}(x,r)$ when there is no confusion. If there is nothing in the superscripts we always mean the ball $B := B(x,r)$ is taken with respect to the distance $\varrho$.

If $|x'| \leq 4r$, then we have the following decomposition of the ball $B(x,r)$.
\begin{align}
\label{Decomposition of balls for small radius}
    B((x', x''),r) \subseteq B^{|\cdot|}(x', r) \times B^{|\cdot|}(x'', C r^2) ,
\end{align}
where $B^{|\cdot|}(x', r)$ denotes the ball of radius $r$ and centered at $x'$ with respect to Euclidean distance and $C>0$ is a constant. For the proofs of the above results see \cite{Martini_Sikora_Grushin_Weighted_Plancherel_2012}, \cite{Niedorf_Grushin_Multiplier_2022}.

We shall make use of the finite speed of propagation property for the corresponding wave operator $\cos{(t\sqrt{\mathcal{L}})}$ for the operator $\mathcal{L}$ that is, the kernel of the operator $\cos{(t\sqrt{\mathcal{L}})}$  satisfies
\begin{align*}
    \supp K_{\cos{(t\sqrt{\mathcal{L}})}} \subseteq \mathcal{D}_{t} := \{ (x, y) \in \mathbb{R}^d \times \mathbb{R}^d: \varrho(x, y)\leq t\}, \quad \text{for all} \quad t >0 .
\end{align*}
An immediate consequence of this result is following well-known lemma.
\begin{lemma}[\cite{Chen_Lin_Yan_Commutator_Bochner_Riesz_Hermite_2023}]
\label{lemma support of kernel}
    Let $F$ be an even bounded Borel function and $\widehat{F}\in L^1(\mathbb{R})$  with $\supp{\widehat{F}} \subset [-t, t].$ Then we have that 
     \begin{align*}
        K_{F(\sqrt{\mathcal{L}})} \subseteq \mathcal{D}_t  .
    \end{align*}
\end{lemma}

A key ingredient in proof of our \Cref{Theorem: Multiplier for Commutator} is the ``truncated restriction type estimate'' for the multiplier operator $ F(\mathcal{L})$. To be precise, we need to further truncate the multiplier $F$ along the spectrum of the Laplacian $T:= (-\Delta_{x''})^{1/2}$. To this end, let us set some notation first. 

Suppose $F : \mathbb{R} \to \mathbb{C}$ is a bounded Borel function supported in $ [R/8, 8R]$ for some $R>0$. Also, suppose $\Theta : \mathbb{R} \to [0,1]$ be an even $C_c^{\infty}$ function supported in $[-2,-1/2] \cup [1/2, 2]$ such that 
\begin{align}
\label{Definition: Cutoff function chi}
    \sum_{M \in \mathbb{Z}} \Theta_M(\kappa) =1 \quad \text{for}\quad \kappa \neq 0,
\end{align}
where $\Theta_M(\kappa) = \Theta(2^{-M} \kappa)$.  For $M \in \mathbb{N}$, let $F_M : \mathbb{R} \times \mathbb{R} \to \mathbb{C}$ be given by 
\begin{align*}
    F_M(\kappa, r) = F(\sqrt{\kappa}) \Theta_{M}(\kappa/r) \quad \text{for} \quad \kappa \geq 0, r \neq 0
\end{align*}
and $F_M(\kappa, r) = 0$ else.

Note that as shown in \cite{Martini_Sikora_Grushin_Weighted_Plancherel_2012}, \cite{Niedorf_Grushin_Multiplier_2022}, the joint functional calculus of the operators $\mathcal{L}$ and $T$ allows us to define the operator $F_M (\mathcal{L}, T)$ for every Borel function $F_M : \mathbb{R} \times \mathbb{R} \rightarrow \mathbb{C}$.

In the sequel, the following explicit formula for the integral kernel of the operator $F_M (\mathcal{L}, T)$ play some important roles. For more details, see Lemma 3.1 in \cite{Niedorf_Grushin_Multiplier_2022}.
\begin{proposition}
\label{reprsentation of kernel}
    Let $K_{F_M (\mathcal{L}, T)}$ be the integral kernel of the operator $F_M (\mathcal{L}, T)$. Then for almost all $x, y \in \mathbb{R}^n$,
\begin{align*}
K_{F_M (\mathcal{L}, T)} (x,y) = \frac{1}{(2\pi)^{d_2}}\, \int_{\mathbb{R}^{d_2}} e^{i(x'' - y'') \cdot \lambda} \sum_{k=0} ^{\infty} F_{M} ((2k+d_1)|\lambda|, |\lambda| )  \mathcal{K}_k^{\lambda}(x',y') \  d\lambda ,
\end{align*}
where $\displaystyle{\mathcal{K}_k^{\lambda}(x',y') = \sum_{|\mu| =k} \Phi_{\mu}^{\lambda}(x') \Phi_{\mu}^{\lambda}(y') }$.
\end{proposition}

The following theorem gives us the restriction type estimate for the operator $\mathcal{L}$.
\begin{proposition}
\label{Theorem: Restriction Estimate}
Suppose $1 \leq p \leq \min{\{2d_1/(d_1 + 2), (2d_2 +2)/(d_2 + 3)\}}$. Let  $F$ and $F_M$ be as above. Then
\begin{align}
\label{Inequality: Truncated restriction equation 1}
    \| F_M(\mathcal{L}, T) f\|_{L^2 } \leq C R^{(d_1 + 2d_2)(1/p - 1/2) } 2^{-Md_2(1/p-1/2)} \|\delta_{R} F\|_{L^2(\mathbb{R})} \, \|f\|_{L^p} .
\end{align}
In particular for $l \in \mathbb{N}$,
\begin{align}
\label{Inequality: Truncated restriction equation 2}
    \left\| \sum_{M>l} F_M(\mathcal{L}, T) f \right\|_{L^2 } \leq C R^{(d_1 + 2d_2)(1/p - 1/2) } 2^{-ld_2(1/p-1/2)} \|\delta_{R} F\|_{L^2(\mathbb{R})} \, \|f\|_{L^p} .
\end{align}
Moreover, for $|y'| > 4s >0 $, 
\begin{align}
\label{Inequality: Truncated restriction equation 3}
    \| F(\sqrt{\mathcal{L}}) P_{B(y,s)} f\|_{ L^2} \leq C R^{(d_1 + d_2)(1/p - 1/2)  }  |y'|^{ - d_2(1/p - 1/2)} \|\delta_{R} F\|_{L^2(\mathbb{R})}\, \|f\|_{L^p}.
\end{align}

\end{proposition}

\begin{proof}
From (\ref{Effect of dilation}) we have
\begin{align*}
    (\delta_R F)(\sqrt{\mathcal{L}})(f \circ \delta_{R^{-1}}) &= (F(\sqrt{\mathcal{L}})f) \circ \delta_{R^{-1}} \\
    \text{and} \quad (\delta_R F)_{M}(\mathcal{L}, T)(f \circ \delta_{R^{-1}}) &= ( F_{M}(\mathcal{L}, T)f) \circ \delta_{R^{-1}} .
\end{align*}
Note that $\delta_R F$ is supported in $[1/8,8]$. Then from Theorem 3.4 of \cite{Niedorf_Grushin_Multiplier_2022} we get
\begin{align*}
    \| F_M(\mathcal{L}, T) f\|_{L^2 } &= R^{-Q/2} \| ( F_{M}(\mathcal{L}, T)f) \circ \delta_{R^{-1}}\|_{L^2 } \\
    &= R^{-Q/2} \| (\delta_R F)_{M}(\mathcal{L}, T)(f \circ \delta_{R^{-1}})\|_{L^2 } \\
    &\leq C R^{-Q/2} 2^{-Md_2(1/p-1/2)} \|\delta_{R} F\|_{L^2(\mathbb{R})} \, \|f \circ \delta_{R^{-1}}\|_{L^p} \\
    &\leq C R^{(d_1 + 2d_2)(1/p - 1/2) } 2^{-Md_2(1/p-1/2)} \|\delta_{R} F\|_{L^2(\mathbb{R})} \, \|f\|_{L^p} .
\end{align*}
Also the proof of (\ref{Inequality: Truncated restriction equation 2}) easily follows from (\ref{Inequality: Truncated restriction equation 1}).

Similarly use of (\ref{Effect of dilation}) and Theorem 3.4 of \cite{Niedorf_Grushin_Multiplier_2022} yields for $|y'| > 4s >0 $,
\begin{align*}
    \| F(\sqrt{\mathcal{L}}) P_{B(y,s)} f\|_{ L^2} &= R^{-Q/2} \| (\delta_R F)(\sqrt{\mathcal{L}}) P_{B(\delta_R y,R s)} (f \circ \delta_{R^{-1}})\|_{ L^2} \\
    &\leq C R^{-Q/2} (R |y'|)^{ - d_2(1/p - 1/2)} \|\delta_{R} F\|_{L^2(\mathbb{R})}\, \|f \circ \delta_{R^{-1}}\|_{L^p} \\
    &\leq C R^{(d_1 + d_2)(1/p - 1/2)  }  |y'|^{ - d_2(1/p - 1/2)} \|\delta_{R} F\|_{L^2(\mathbb{R})}\, \|f\|_{L^p}.
\end{align*}
    
\end{proof}

Another important tool for the proof of \Cref{Theorem: Multiplier for Commutator} is the weighted Plancherel estimate for the operator  $F_M (\mathcal{L}, T)$.  
\begin{proposition}
\label{Theorem: Weighted Plancherel in second variable}
  Let $F$ and $F_M$ be as above.  Then for all $N \in \mathbb{N} \cup \{0\}$ and almost all $y \in \mathbb{R}^d$,
\begin{align*}
    \left( \int_{\mathbb{R}^d} \left| |x''-y''|^N K_{F_M(\mathcal{L}, T)}(x,y)  \right|^2 \ dx  \right)^{1/2} &\leq C R^{-2N} R^{Q/2} 2^{M(N-d_2/2)} \|\delta_R F \|_{L^2_{N}} .
\end{align*}
\end{proposition}

\begin{proof}
    As $\supp (\delta_R F) \subseteq [1/8, 8]$, from Lemma 11 of \cite{Martini_Muller_Multiplier_Grushin_2014}, for almost all $y \in \mathbb{R}^d$ we have
    \begin{align*}
        \left( \int_{\mathbb{R}^d} \left| |x''-y''|^N K_{(\delta_R F)_M(\mathcal{L}, T)}(x,y)  \right|^2 \ dx  \right)^{1/2} &\leq C 2^{M(N-d_2/2)} \|\delta_R F \|_{L^2_{N}} .
    \end{align*}

As the above expression is true for almost all $y \in \mathbb{R}^d$, in particular we can take $\delta_R y$ (see preliminaries section for definition). From \Cref{reprsentation of kernel} and making change of variables $x$ goes to $\delta_R x$ and $\lambda$ goes to $\lambda/R^2$, we get
\begin{align*}
     & \int_{\mathbb{R}^d} \left| |x''-R^2 y''|^N K_{(\delta_R F)_M(\mathcal{L}, T)}(x, \delta_R y)  \right|^2 \ dx \\
     &= R^Q \frac{1}{(2\pi)^{2 d_2}} \int_{\mathbb{R}^d} \Bigl| (R^2 |x''-y''|)^N \, \int_{\mathbb{R}^{d_2}} e^{i R^2(x'' - y'') \cdot \lambda} \sum_{k=0} ^{\infty} F(R\sqrt{(2k+d_1)|\lambda|}) \Theta(2^{-M} (2k+d_1)) \\
     & \hspace{10cm} \sum_{|\mu|= k} \Phi_{\mu} ^{\lambda} (R x') \Phi_{\mu} ^{\lambda} (R y')\, d\lambda  \Bigr|^2 \ dx \\
     &= R^{-Q} \frac{1}{(2\pi)^{2 d_2}} \int_{\mathbb{R}^d} \Bigl| (R^2 |x''-y''|)^N \, \int_{\mathbb{R}^{d_2}} e^{i (x'' - y'') \cdot \lambda} \sum_{k=0} ^{\infty} F(\sqrt{(2k+d_1)|\lambda|}) \Theta(2^{-M} (2k+d_1)) \\
     & \hspace{10cm} \sum_{|\mu|= k} \Phi_{\mu} ^{\lambda} ( x') \Phi_{\mu} ^{\lambda} ( y')\, d\lambda  \Bigr|^2 \ dx ,
\end{align*}
from which the required result follows immediately.
    
\end{proof}

Let us consider the first order gradient vector fields:
\begin{align*}
    X_j = \frac{\partial}{\partial x_j'} \quad \text{and} \quad X_{j,k} = x_j' \frac{\partial}{\partial x_k''} \quad \text{for} \ 1\leq j \leq d_1 \ \text{and} \ 1\leq k \leq d_2 .
\end{align*}
Also set $X = (X_j, X_{j,k})_{1\leq j \leq d_1, 1\leq k \leq d_2}$. Let us recall the `Mean value estimate' which will be used later.
\begin{lemma}[\cite{Bagchi_Basak__Garg_Ghosh_Sparsh_bound_pseudo_multiplier_Grushin_2023}]
\label{Theorem: Mean value theorem}
    There exists constants $C_1, C_2>0$ such that for any ball $B(x_0, r)$ and points $x,y \in B(x_0,r)$, there exists a $\Tilde{\varrho}$-length minimizing curve $\gamma_0 : [0,1] \to B(x_0, C_1 r)$ joining $x$ to $y$, and for any $f \in C^1(B(x_0, C_1 r))$,
    \begin{align*}
        |f(x)-f(y)| &\leq C_2 \varrho(x,y) \int_0^1 |X f(\gamma_0(s))| \ ds .
    \end{align*}
\end{lemma}

Now we are going to state the pointwise version of weighted Plancherel theorem.
\begin{proposition}
\label{Theorem: Weighted Plancherel with L infinity condition}
Let $\Gamma \in \mathbb{N}^{d_1 +d_1 d_2}$. For all $\beta \geq 0$ with $s> \beta + 1/2$ and all bounded Borel functions $F : \mathbb{R} \to \mathbb{C}$ supported in $[0,R]$,
\begin{align*}
    \esssup_y \left|(1+R \varrho(x,y))^{\beta} K_{F(\sqrt{\mathcal{L}})}(x,y) \right| &\leq C |B(x, R^{-1})|^{-1} \|\delta_R F\|_{L^{\infty}_{s}} \\
   \text{and} \quad \esssup_y \left|(1+R \varrho(x,y))^{\beta} X_x^{\Gamma} K_{F(\sqrt{\mathcal{L}})}(x,y) \right| &\leq C R^{|\Gamma|} \  |B(x, R^{-1})|^{-1} \|\delta_R F\|_{L^{\infty}_{s}} .
\end{align*}
\end{proposition}

\begin{proof}
For all $t>0$ the integral kernel $p_t$ of the operator $\exp{(-t\mathcal{L})}$ satisfies (see \cite{Martini_Sikora_Grushin_Weighted_Plancherel_2012})
\begin{align*}
    | p_t(x,y)| & \leq C |B(x, t^{1/2})|^{-1} e^{-c \varrho(x,y)^2/t} ,
\end{align*}
for all $x,y \in \mathbb{R}^d$.
Note that from \cite{Dziubanski_Sikora_Lie_group_Grushin_2021}, (also see \cite{Bagchi_Basak__Garg_Ghosh_Sparsh_bound_pseudo_multiplier_Grushin_2023}) we have
\begin{align*}
    |X_x^{\Gamma} p_t(x,y)| & \leq C t^{-|\Gamma|/2} |B(x, t^{1/2})|^{-1} e^{-c \varrho(x,y)^2/t}  .
\end{align*}
Therefore following the proof of Lemma 4.3 of \cite{Bui_Duong_Spectral_multiplier__Trieble_Lizorkin_2021} and using Theorem 7.3 of \cite{Ouhabaz_Analysis_heat_equation_domain_2005} we get our required estimates.
\end{proof}

The following lemma will be used later in our proof.
\begin{lemma}
\label{lemma: outside distance}
    Let $R> 0$. Then for any $M> Q$ we have
    \begin{align*}
        \int_{\varrho(x, y) \geq r} \frac{dy}{\big( 1 + R\varrho(x, y)   \big)^M} \lesssim R^{-M} r^{-M + d_1 + d_2}\, \max\{r, |x'|\}^{d_2}.
    \end{align*}
\end{lemma}
\begin{proof}
Decomposing the integral into annular region we have
\begin{align*}
  \int_{\varrho(x, y) \geq r} \frac{dy}{\big( 1 + R\varrho(x, y)   \big)^M} &= \sum_{k = 0} ^\infty \int_{2^k r \leq \varrho(x, y) \leq 2^{k+ 1} r}  \frac{dy}{\big( 1 + R \varrho(x, y)   \big)^M} \\
 & \leq C \sum_{k = 0} ^\infty \frac{R^{-M}}{(2^k r)^M} (2^k r)^{d_1+d_2} \max\{(2^k r), |x'|\}^{d_2} \\
 & \leq C R^{-M} r^{-M + d_1 + d_2} \max\{r, |x'|\}^{d_2} ,
\end{align*}
provided $M>Q$. 
\end{proof}

Next, we recall the notion of the
space of functions of bounded mean oscillation ($BMO$). A function $f \in L^1 _{loc} (\mathbb{R}^d)$ is said to be in the space $BMO^{\varrho}(\mathbb{R}^d)$ if 
\begin{align}
\label{Definition of BMO}
    \|f\|_{BMO^{\varrho}(\mathbb{R}^d)}:= \sup_{B} \frac{1}{|B|} \int_{B} |f(x) - f_B| \, dx < \infty.
\end{align}
One of the fundamental properties of functions with bounded mean oscillation is the celebrated John–Nirenberg inequality  which says that a function with bounded mean oscillation has exponential decay of its distribution function. An immediate consequence of this distribution inequality is the $L^p$ characterization of $BMO^{\varrho}(\mathbb{R}^d)$ norms:
\begin{align}
\label{Cherecterisation of BMO}
    \|f\|_{BMO^{\varrho}(\mathbb{R}^d)} \sim \sup_{B}\left(\frac{1}{|B|}\int_{B}|f(x)-f_B|^p \, dx \right)^{\frac{1}{p}},
\end{align}
for any $1<p<\infty$. 

Moreover, we also need the following inequality. Let $f$ be in $BMO^{\varrho}(\mathbb{R}^d)$. Given a ball $B$ and a positive integer $m$ we have
\begin{align}
\label{Difference of average in terms of BMO}
    |b_{B}- b_{2^m B}| &\leq 2^Q m \|f\|_{BMO^{\varrho}(\mathbb{R}^d)}.
\end{align}
One can easily see that $L^{\infty}(\mathbb{R}^d)$ is a proper subspace of $BMO^{\varrho}(\mathbb{R}^d)$, for example the unbounded function $\log (\varrho(x,0))$ is in $BMO^{\varrho}(\mathbb{R}^d)$. For the proof the above results related to $BMO^{\varrho}(\mathbb{R}^d)$ functions see \cite{Ding_lee_Lin_Hardy_BMO_General_Sets_2014}, \cite{Duong_Yan_BMO_John_Nirenber_2005}, \cite{Chen_Tian_Ward_Commutator_Bochner_Riesz_Elliptic_2021}, \cite{Grafakos_Modern_Fourier_Analysis_2014}.

To prove the compactness of the Bochner-Riesz commutators $[b, S^{\alpha}]$,  we need the following Kolmogorov-Riesz compactness theorem.
\begin{theorem}[\cite{Olsen_Holden_Kolmogorov_Riesz_Compactness_2010}]
\label{Theorem: Characterizations of compactness}
Let $1 < p < \infty$. The subset $\mathcal{F}$ in $L^p(\mathbb{R}^d)$ is called totally bounded if $\mathcal{F}$ satisfies the following conditions:
\begin{enumerate}
       \item  Norm boundedness uniformly : $\displaystyle{ \sup_{f \in \mathcal{F}} \|f\|_{L^p (\mathbb{R}^d)} < \infty}$;
       \item Equicontinuous uniformly: $\displaystyle{\lim_{t \rightarrow 0} \| f(\cdot + t) - f(\cdot)\|_{L^p (\mathbb{R}^d)} = 0   }$ uniformly in $f\in \mathcal{F}$;
       \item  Control uniformly away from origin: $\displaystyle{\lim_{A \rightarrow \infty} \|\chi_{\mathcal{F}_A} f \|_{L^p (\mathbb{R}^d)}=0}$, uniformly in $f\in \mathcal{F}$, where $\chi_{\mathcal{F}_A} = \{x \in \mathbb{R}^d : \varrho(x, 0) > A\}$.
\end{enumerate}
\end{theorem}

\section{Boundedness of Bochner-Riesz Commutator}

\begin{proof}[Proof of \Cref{Theorem: Bochner-Riesz Commutator}]
Assuming the proof of \Cref{Theorem: Multiplier for Commutator} one can be easily prove \Cref{Theorem: Bochner-Riesz Commutator}. Indeed take $F(\eta)= \left(1-\eta^2 \right)^{\alpha}_{+}$ and $t= 1$. Then one easily check that $F \in W^2_{\beta}$ if and only if $\alpha> \beta -1/2$. Therefore from \Cref{Theorem: Multiplier for Commutator} for $\alpha>d(1/p-1/2)-1/2$ and $p<q<p'$ we get
\begin{align*}
    \|[b, S^{\alpha}(\mathcal{L})]f\|_{L^q} &= \|[b, F(t \sqrt{\mathcal{L}})]f\|_{L^q} \leq C \|b\|_{BMO^{\varrho}(\mathbb{R}^d)} \|F\|_{W^2_{\beta}} \|f\|_{L^q} \leq C \|b\|_{BMO^{\varrho}(\mathbb{R}^d)} \|f\|_{L^q} .
\end{align*}    
\end{proof}

Next we prove the \Cref{Theorem: Multiplier for Commutator}.

\begin{proof}[Proof of \Cref{Theorem: Multiplier for Commutator}]

We proceed in several steps. Let $f \in C_c^{\infty}(\mathbb{R}^d)$.

$(1)$ \emph{Reduction to the case $p<q<2$:} Observe that 
\begin{align*}
    [b, F(t\sqrt{\mathcal{L}})]^{*} &= -[\Bar{b}, \Bar{F}(t\sqrt{\mathcal{L}})].
\end{align*}
Therefore, with the aid of duality and interpolation argument, \Cref{Theorem: Multiplier for Commutator} is reduced to proving only for the indices $p<q<2$.

$(2)$ \emph{Decomposition of $F$ into sum of dyadic parts in the Fourier transform side:} Let us fix an even function $\eta \in C_c ^{\infty} (\mathbb{R})$ such that $\supp{\eta} \subseteq \{\xi : 1/4\leq |\xi| \leq 1 \}$ and $\sum_{l\in \mathbb{Z}} \eta(2^{-l} \xi) = 1$, for all $\xi \neq 0$.  Set $\eta_0 (\xi) = 1 - \sum_{l\geq 1} \eta(2^{-l} \xi)$ and $\eta_l (\xi) = \eta(2^{-l} \xi), \,l \in \mathbb{N}$. Now we define 
\begin{align*}
    F^{(l)} (u) = \frac{1}{2 \pi} \int_{\mathbb{R}} \eta_l (\xi) \widehat{F}(\xi) \cos{u\xi} \, d\xi, \quad l \geq 0.
\end{align*}
Then by spectral theorem we decompose  $ F(t\sqrt{\mathcal{L}})$ as
\begin{align}
\label{Equation: Dyadic decomposition}
    F(t\sqrt{\mathcal{L}}) &= \sum_{l \geq 0} F^{(l)} (t\sqrt{\mathcal{L}}).
\end{align}
Therefore
\begin{align*}
    \|[b, F(t\sqrt{\mathcal{L}})] f\|_{L^q} & \leq \sum_{l\geq 0} \|[b, F^{(l)}(t\sqrt{\mathcal{L}})f]\|_{L^q}.
\end{align*}
This shows that, in order to prove (\ref{main inequality}),  it suffices to show that 
\begin{equation}
\label{equation: inequality for F^l}
  \|[b, F^{(l)}(t\sqrt{\mathcal{L}})f]\|_{L^q(\mathbb{R}^d)} \leq C 2^{-\varepsilon \,l} \|b\|_{BMO^{\varrho}(\mathbb{R}^d)} \|F\|_{W^2_{\beta}(\mathbb{R})} \|f\|_{L^q(\mathbb{R}^d)} \quad \mbox{for}\,\,\mbox{some}\,\, \varepsilon>0.
\end{equation}

$(3)$ \emph{Reduction of boundedness of the commutator to that of corresponding operator:} We use the idea from \cite{Chen_Ouhabaz_Bochner-Riesz_Grushin_2016}. Note that $\mathbb{R}^d$ is separable with respect to the distance $\varrho$. Therefore we can choose a sequence $\{x_n\}_{n \in \mathbb{N}}$ in $\mathbb{R}^d$ such that $\varrho(x_i, x_j) > \frac{2^l t}{10}$ for $i\neq j$ and $\sup_{x \in \mathbb{R}^d} \inf_i \varrho(x, x_i) \leq \frac{2^l t}{10}$. For $\kappa>0$ let $\kappa B_i := B(x_i, \kappa 2^l t)$. Therefore we may find a decomposition of disjoint sets $\widetilde{B}_i \subseteq B\left(x_i, \frac{2^l t}{10}\right)$ such that for any $\kappa>0$, number of overlapping disjoint balls $\kappa B_i$ can be bounded by $C(\kappa)$, which is independent of $l$ and $t$. One can check that for $i \neq j$, $B(x_i, \frac{2^l t}{20}) \bigcap B(x_j, \frac{2^l t}{20})  = \emptyset$.

Now we can see that 
\begin{eqnarray*}
    \mathcal{D}_{2^l t} \subseteq \bigcup_{\{i,j : \varrho(x_i, x_j) < 2^{l+1}t\}} \widetilde{B}_i \times \widetilde{B}_j \subseteq \mathcal{D}_{2^{l+2} t}.
\end{eqnarray*}
Therefore we get
\begin{eqnarray*}
F^{(l)} (t\sqrt{\mathcal{L}})f = \sum_{\{i,j : \varrho(x_i, x_j) < 2^{l+1} t\}} P_{\widetilde{B}_i} F^{(l)} (t\sqrt{\mathcal{L}})P_{\widetilde{B}_j} f   . 
\end{eqnarray*}
Using \Cref{lemma support of kernel}, one can check that $\supp K_{[b,F^{(l)} (t\sqrt{\mathcal{L}})]} \subseteq \mathcal{D}_{2^l t}$. Then we have 

\begin{align*}
    [b, F^{(l)} (t\sqrt{\mathcal{L}})] f =  \sum_{\{i,j : \varrho(x_i, x_j) < 2^{l+1} t\}} P_{\widetilde{B}_i}  [b, F^{(l)} (t\sqrt{\mathcal{L}})]P_{\widetilde{B}_j} f  .
\end{align*}
Using $[b, F^{(l)} (t\sqrt{\mathcal{L}})] = (b - b_{B_j}) F^{(l)} (t\sqrt{\mathcal{L}}) - F^{(l)} (t\sqrt{\mathcal{L}})(b - b_{B_j})$, we get
\begin{align*}
   & \|[b, F^{(l)} (t\sqrt{\mathcal{L}})]f\|_{L^q}^q \\
   & \leq C K^{q-1} \sum_i \sum_{\{j : \varrho(x_i, x_j) < 2^{l+1} t\}} \Bigl(\| P_{\widetilde{B}_i}\, (b - b_{B_j})\, F^{(l)} (t\sqrt{\mathcal{L}}) P_{\widetilde{B}_j} f \|_{L^q}^q \\
   & \hspace{7cm} + \| P_{\widetilde{B}_i}\, F^{(l)} (t\sqrt{\mathcal{L}}) (P_{\widetilde{B}_j}\, (b - b_{B_j}) f )\|_{L^q}^q  \Bigr) \\
   & \leq C K^{q} \sum_j \Bigl(\| P_{4 B_j}\, (b - b_{B_j})\, F^{(l)} (t\sqrt{\mathcal{L}}) P_{\widetilde{B}_j} f \|_{L^q}^q \\
   & \hspace{7cm} + \| P_{4 B_j}\, F^{(l)} (t\sqrt{\mathcal{L}}) (P_{\widetilde{B}_j}\, (b - b_{B_j}) f )\|_{L^q}^q  \Bigr) \\
   & := S_1 + S_2 .
\end{align*}
We estimate each term separately.

\underline{\emph{Estimation of $S_1$}:}
Since $F^{(l)}$ is not compactly supported, we take a bump function $\psi \in C_c ^{\infty}(\mathbb{R})$  with support in $[1/16, 4]$ and $\psi = 1$ in $[1/8, 2]$. Then, it follows that 
\begin{align*}
    S_1^{1/q} & \leq C \left( \sum_j \|P_{4 B_j}\, (b - b_{B_j})\, (\psi F^{(l)})(t\sqrt{\mathcal{L}}) P_{\widetilde{B}_j} f\|_{L^q}^q \right)^{1/q} \\
    & \hspace{5cm} + C \left( \sum_j \|P_{4 B_j}\, (b - b_{B_j})\, ((1-\psi) F^{(l)}) (t\sqrt{\mathcal{L}}) P_{\widetilde{B}_j} f\|_{L^q}^q \right)^{1/q} \\
    & := S_{11} + S_{12}.
\end{align*}
Therefore, in order to estimate $S_1$, it suffices to estimate each $S_{11}$ and $S_{12}$ separately. First, we focus our attention to $S_{11}$.

\underline{\emph{Estimation of $S_{11}$}:}
 As measure of a ball depends also on the center of that ball, we consider the following two cases.

\underline{Case I:} Assume $|x_j'| \geq 2^{l+2} t$.

An application of H\"older's inequality implies that
\begin{align*}
     \|P_{4 B_j}\, (b - b_{B_j})\, (\psi F^{(l)})(t\sqrt{\mathcal{L}}) P_{\widetilde{B}_j} f\|_{L^q} &\leq C \|P_{4 B_j}\, (b - b_{B_j}) \|_{L^{\frac{2q}{2-q}}} \| (\psi F^{(l)})(t\sqrt{\mathcal{L}}) P_{\widetilde{B}_j} f \|_{L^2} .
\end{align*}
Notice that, by (\ref{Cherecterisation of BMO}) and (\ref{Difference of average in terms of BMO}) we get
\begin{align}
\label{Inequality: BMO norm calculation} 
   \|P_{4 B_j}\, (b - b_{B_j})\|_{L^{\frac{2q}{2-q}}} &= |4B_j|^{\frac{2-q}{2q}}\, \left( \frac{1}{|4B_j|} \int_{4 B_j} |b - b_{B_j}|^{\frac{2q}{2-q}}   \right)^{ \frac{2-q}{2q}} \\
   &\nonumber \leq C |4B_j|^{\frac{2-q}{2q}}\, \left( \frac{1}{|4B_j|} \int_{4B_j} |b - b_{4B_j}|^{\frac{2q}{2-q}}   \right)^{ \frac{2-q}{2q}} \\
   & \hspace{6cm} \nonumber + C |4B_j|^{\frac{2-q}{2q}} |b_{4B_j} - b_{B_j}| \\
   &\nonumber \leq C |4B_j|^{\frac{2-q}{2q}} \|b\|_{BMO^{\varrho}(\mathbb{R}^d)} .
 \end{align}
Applying this estimate together with \Cref{Theorem: Restriction Estimate} and H\"older's inequality, we see that
\begin{align*}    
     & \|P_{4 B_j}\, (b - b_{B_j})\, (\psi F^{(l)})(t\sqrt{\mathcal{L}}) P_{\widetilde{B}_j} f\|_{L^q} \\
     & \leq C \|b\|_{BMO^{\varrho}(\mathbb{R}^d)}\, |B(x_j, 2^{l+2} t)|^{1/q - 1/2}\, \|  (\psi F^{(l)}) (t\sqrt{\mathcal{L}}) P_{\widetilde{B}_j} f \|_{L^2}\\
     & \leq C \|b\|_{BMO^{\varrho}(\mathbb{R}^d)}\, (2^l t)^{(d_1 + d_2) (1/q - 1/2)} |x_j'|^{d_2(1/q - 1/2)} \, t^{- (d_1 + d_2) (1/p - 1/2)}\,  \\
     & \hspace{7cm} |x_j'|^{-d_2 (1/p -1/2)}\, \|\delta_{t^{-1}}(\psi F^{(l)})(t \cdot)\|_{L^2} \|P_{\widetilde{B}_j} f\|_{L^p}\\
     & \leq C \|b\|_{BMO^{\varrho}(\mathbb{R}^d)}\, (2^l t)^{(d_1 + d_2) (1/q - 1/2)} |x_j'|^{d_2(1/q - 1/2)} \, t^{- (d_1 + d_2) (1/p - 1/2)}  |x_j'|^{-d_2 (1/p -1/2)}\, \\
     & \hspace{6cm}  \|\psi F^{(l)}\|_{L^2} \, (2^l t)^{(d_1 + d_2) (1/p - 1/q)} |x_j'|^{d_2(1/p - 1/q)} \|P_{\widetilde{B}_j} f\|_{L^q}\\
     & \leq C \|b\|_{BMO^{\varrho}(\mathbb{R}^d)}\, 2^{l(d_1 +d _2) (1/p - 1/2)}\, 2^{-\beta l}\|F\|_{W^2_{\beta}} \|P_{\widetilde{B}_j} f\|_{L^q} .
\end{align*}
Therefore, for $\beta>d(1/p-1/2)$, we have 
\begin{align*}
    S_{11} &\leq C  2^{-l \varepsilon} \|b\|_{BMO^{\varrho}(\mathbb{R}^d)} \|F\|_{W^2_{\beta}} \|f\|_{L^q},
\end{align*}
for some $\varepsilon > 0$.

\underline{Case II:} Assume $|x_j'| < 2^{l+2} t$.

Let $\Theta$ be as in (\ref{Definition: Cutoff function chi}). Define the function $F^{(l)}_M : \mathbb{R} \times \mathbb{R} \to \mathbb{C}$ by setting
\begin{align*}
    F^{(l)}_M (\kappa, r) &= (\psi F^{(l)}) (t\sqrt{\kappa}) \,\Theta_{M}(\kappa /r)  \quad \text{for} \quad \kappa \geq 0, r \neq 0
\end{align*}
and $F^{(l)}_M (\kappa , r) = 0$ else. Then we can write 
\begin{align}
\label{Equation: Truncation along spectrum for S_{11}}
    P_{4 B_j}\, (b - b_{B_j})\, (\psi F^{(l)})(t\sqrt{\mathcal{L}}) P_{\widetilde{B}_j} f &=  P_{4 B_j}\, (b - b_{B_j}) \, \left( \sum_{M=0 } ^l + \sum_{M=l+ 1 } ^{\infty}\right) F^{(l)}_M (\mathcal{L}, T) P_{\widetilde{B}_j} f \\
    & \nonumber =: g_{j, \leq l}  + g_{j, > l} .
\end{align}
Note that we have $d_1 \geq 2$, because of the range of $p$. Thus $\Theta_{M}([k])=0$ for $M \leq 0$. The second term $g_{j,>l}$ can be easily controlled by using H\"older's inequality, (\ref{Inequality: BMO norm calculation}), \Cref{Theorem: Restriction Estimate} and again using H\"older's inequality. Indeed
\begin{align}
\label{Estimate: g_{j,>l} for S_{11}}
   & \| g_{j, > l}\|_{L^q}\\
   & \nonumber \leq  \|P_{4 B_j}\, (b - b_{B_j})\|_{L^{\frac{2q}{2-q}}}  \left\|\sum_{M=l+ 1 } ^{\infty} F^{(l)}_M (\mathcal{L}, T) P_{\widetilde{B}_j} f  \right\|_{L^2} \\
   & \nonumber \leq C \|b\|_{BMO^{\varrho}(\mathbb{R}^d)} (2^l t) ^{Q(1/q - 1/2)}\, t^{-Q(1/p - 1/2)}\, 2^{-ld_2 (1/p - 1/2)}\, \|\delta_{t^{-1}}(\psi F^{(l)})(t \cdot)\|_{L^2}\, \|P_{\widetilde{B}_j} f\|_{L^p} \\
   & \nonumber \leq C \|b\|_{BMO^{\varrho}(\mathbb{R}^d)} (2^l t) ^{Q(1/q - 1/2)}\, t^{-Q(1/p - 1/2)}\, 2^{-ld_2 (1/p - 1/2)}\, \|\psi F^{(l)}\|_{L^2}\, (2^lt)^{Q(1/p - 1/q)}\, \|P_{\widetilde{B}_j} f\|_{L^q}\\
   & \nonumber \leq C \|b\|_{BMO^{\varrho}(\mathbb{R}^d)} 2^{ld(1/p - 1/2)}\,2^{-\beta l} \|F\|_{W^2_{\beta}} \, \|P_{\widetilde{B}_j} f\|_{L^q}.
\end{align}
It remains to estimate the term $g_{j, \leq l}$. As we have assumed $|x_j'| < 2^{l+2} t$, notice that in this case, from (\ref{Decomposition of balls for small radius}) we have $\widetilde{B}_j \subseteq B(x_j, 2^l t) \subseteq B^{|\cdot|} (x_j', 2^l t) \times B^{|\cdot|} (x_j'', C 2^{2l} t^2)$.

For each $0 \leq M \leq l$, we decompose 
\begin{align*}
    \widetilde{B}_j = \bigcup_{m = 1}^{N_M} \widetilde{B}_{j, m} ^M ,
\end{align*}
where $\widetilde{B}_{j, m} ^M \subseteq B^{|\cdot|} (x_j', 2^l t) \times B^{|\cdot|} (x_{j,m}^M, C 2^Mt  \cdot 2^{l} t) $ are disjoint subsets and  
$|x_{j,m}^M- x_{j,m'}^M| >2^M t \cdot  2^{l} t/2 $ for $m \neq m'$. One can see that $N_M \leq C 2^{(l-M)d_2}$. Moreover, for given $\gamma>0$ the number $N_{\gamma}$ of overlapping ball 
\begin{align*}
    \widetilde{B_{j, m} ^M} := B^{|\cdot|} (x_j', 2^{l+2} t) \times B^{|\cdot|} (x_{j,m}^M, C 2^{\gamma l +1} 2^Mt  \cdot 2^{l} t), \quad 1 \leq m \leq N_M 
\end{align*}
can be bounded by $N_{\gamma} \leq C 2^{C\gamma l}$. 

With the aid of this decomposition, we further decompose $f$ as $\displaystyle{P_{\widetilde{B}_j} f = \sum _{m = 1} ^{N_M} P_{\widetilde{B}_{j, m} ^M} f }$. Let us set $g_{j, m} ^M := P_{4 B_j}\, (b - b_{B_j}) \, F^{(l)}_M (\mathcal{L}, T) P_{\widetilde{B}_{j, m} ^M} f$ and write 
\begin{align*}
    g_{j, \leq l} &= \sum_{M= 0} ^l \sum_{m= 1} ^{N_M} \chi_{\widetilde{B_{j, m} ^M}} g_{j, m} ^M + \sum_{M= 0} ^l \sum_{m= 1} ^{N_M}(1 - \chi_{\widetilde{B_{j, m} ^M}}) g_{j, m} ^M ,
\end{align*}
where $\chi_{\widetilde{B_{j, m} ^M}}$ is the characteristic function of the set $\widetilde{B_{j, m} ^M}$.

Let us estimate the first term. In view of H\"older's inequality and bounded overlapping property of the balls $\widetilde{B_{j, m} ^M}$, we obtain that
\begin{align}
 \label{inequality of first term in S_{11}} \left\|  \sum_{M= 0} ^l \sum_{m= 1} ^{N_M} \chi_{\widetilde{B_{j, m} ^M}} g_{j, m} ^M \right\|_{L^q} ^q & \leq C ((l+ 1)N_{\gamma})^{q-1} \sum_{M= 0} ^l \sum_{m= 1} ^{N_M} \left\|  \chi_{\widetilde{B_{j, m} ^M}} g_{j, m} ^M \right\|_{L^q} ^q .
\end{align}
Choose $s$ such that $q<s<2$. Now, using H\"older's inequality together with (\ref{Inequality: BMO norm calculation}) and (\ref{Theorem: Restriction Estimate}) and again H\"older's inequality, we see that
\begin{align*}
    & \left\|  \chi_{\widetilde{B_{j, m} ^M}} g_{j, m} ^M \right\|_{L^q}\\
    & \leq \left\|  P_{4 B_j}\, (b - b_{B_j})  \right\|_{L^{\frac{sq}{s-q}}}\, \left\| \chi_{\widetilde{B_{j, m} ^M}} F^{(l)}_M (\mathcal{L}, T) P_{\widetilde{B}_{j, m} ^M} f \right\|_{L^s}\\
    & \leq C (2^l t)^{Q(1/q - 1/s)} \, \|b\|_{BMO^{\varrho}(\mathbb{R}^d)} (2^l t)^{d_1(1/s - 1/2)}\, (2^{\gamma l} 2^M t \cdot 2^l t)^{d_2(1/s - 1/2)} \left\|  F^{(l)}_M (\mathcal{L}, T) P_{\widetilde{B}_{j, m} ^M} f \right\|_{L^2}\\
    & \leq C (2^l t)^{Q(1/q - 1/s)} \, \|b\|_{BMO^{\varrho}(\mathbb{R}^d)} (2^l t)^{d_1(1/s - 1/2)}\, (2^{\gamma l} 2^M t \cdot 2^l t)^{d_2(1/s - 1/2)}  t^{-Q(1/p - 1/2)}  \\
    & \hspace{8cm} 2^{-Md_2(1/p - 1/2)}\, \|\delta_{t^{-1}}(\psi F^{(l)})(t \cdot)\|_{L^2}\, \|P_{\widetilde{B}_{j, m} ^M} f\|_{L^p}\\
    & \leq C (2^l t)^{Q(1/q - 1/s)} \, \|b\|_{BMO^{\varrho}(\mathbb{R}^d)} (2^l t)^{d_1(1/s - 1/2)}\, (2^{\gamma l} 2^M t \cdot 2^l t)^{d_2(1/s - 1/2)}  t^{-Q(1/p - 1/2)} 2^{-Md_2(1/p - 1/2)}\,  \\
    & \hspace{6cm} \|\psi F^{(l)}\|_{L^2}\, (2^l t)^{d_1(1/p - 1/q)} (2^M t\cdot  2^l t)^{d_2(1/p - 1/q)}\, \|P_{\widetilde{B}_{j, m} ^M} f\|_{L^q}\\
    & \leq C 2^{l\varepsilon} 2^{ld_2(1/q-1/s) + ld(1/p - 1/2)}\, 2^{-Md_2(1/q - 1/s)}\,  \|b\|_{BMO^{\varrho}(\mathbb{R}^d)} \, \|F^{(l)}\|_{L^2} \|P_{\widetilde{B}_{j, m} ^M} f\|_{L^q}\, \\
    & \leq C  2^{ ld(1/p - 1/2) }\, 2^{l\varepsilon}\, 2^{-l\beta} \, \|b\|_{BMO^{\varrho}(\mathbb{R}^d)} \,  \|F\|_{W_{\beta} ^2} \|P_{\widetilde{B}_{j, m} ^M} f\|_{L^q}, 
\end{align*}
for some $\varepsilon>0$, by choosing $s$ very close to $q$ and $\gamma>0$ small enough.

Substituting this estimate into (\ref{inequality of first term in S_{11}}), we get
\begin{align}
\label{Estimate: g_{j, <l} for S_{11} main part}
  &  \left\|  \sum_{M= 0} ^l \sum_{m= 1} ^{N_M} \chi_{\widetilde{B_{j, m} ^M}} g_{j, m} ^M \right\|_{L^q} ^q\\
  & \nonumber \leq C ((l+ 1)N_{\gamma})^{q-1} \sum_{M= 0} ^l \sum_{m= 1} ^{N_M} 2^{ ldq (1/p - 1/2)} 2^{lq\varepsilon}\, 2^{-l\beta q} \|b\|_{BMO^{\varrho}(\mathbb{R}^d)} ^q\,  \|F\|_{W_{\beta} ^2} ^q\, \|P_{\widetilde{B}_{j, m} ^M} f\|_{L^q} ^q\\
  & \nonumber \leq C 2^{ ldq (1/p - 1/2) + lq \varepsilon - l \beta q}\, \|b\|_{BMO^{\varrho}(\mathbb{R}^d)} ^q\, \|F\|_{W_{\beta} ^2} ^q\, \|P_{\widetilde{B}_j} f\|_{L^q} ^q .
\end{align}
We are left with the estimation of the remainder term $\displaystyle{\sum_{M= 0} ^l \sum_{m= 1} ^{N_M}(1 - \chi_{\widetilde{B_{j, m} ^M}}) g_{j, m} ^M }$. Using H\"older's inequality, we get
\begin{align}
\label{Inequality: Error part 1-chi for S_11}
   & \left\| \sum_{M= 0} ^l \sum_{m= 1} ^{N_M}(1 - \chi_{\widetilde{B_{j, m} ^M}}) g_{j, m} ^M  \right\|_{L^q} \\
   &\nonumber \leq \left \|  P_{4B_j}(b - b_{B_j})  \right\|_{L^{\frac{2q}{2-q}}} \left \| P_{4B_j} \sum_{M=0} ^l \sum_{m= 1} ^{N_M} ( 1- \chi_{\widetilde{B_{j, m} ^M}}) F^{(l)} _M (\mathcal{L}, T) P_{\widetilde{B}_{j, m} ^M} f    \right \|_{L^2} .
\end{align}
Now, in order to get our desired estimates, we investigate $L^p \rightarrow L^2$ boundedness property of the operator $\displaystyle{ P_{4B_j} \sum_{M=0} ^l \sum_{m= 1} ^{N_M} ( 1- \chi_{\widetilde{B_{j, m} ^M}}) F^{(l)} _M (\mathcal{L}, T) P_{\widetilde{B}_{j, m} ^M} }$ for $1 \leq p < 2$.

For $L^2 \rightarrow L^2$ boundedness, using Plancherel theorem, H\"older's inequality and upper bound of $N_M$(calculated above) we get 
\begin{align*}
\left \|P_{4B_j} \sum_{M=0} ^l \sum_{m= 1} ^{N_M} ( 1- \chi_{\widetilde{B_{j, m} ^M}}) F^{(l)} _M (\mathcal{L}, T) P_{\widetilde{B}_{j, m} ^M} f    \right \|_{L^2} & \leq C \|\psi F^{(l)}\|_{L^{\infty}} \sum_{M= 0} ^l \sum_{m = 1} ^{N_M}\| P_{\widetilde{B}_{j, m} ^M} f \|_{L^2}\\
& \leq C \|F^{(l)} \|_{L^{2}_{1/2+\varepsilon}} (l+1) \,2^{ld_2/2}\,\|P_{\widetilde{B}_{j} } f \|_{L^2} \\
& \leq C \|F^{(l)} \|_{L^{2}}\, 2^{l\varepsilon} 2^{l(1 + d_2)/2} \, \|P_{\widetilde{B}_{j} } f \|_{L^2},
\end{align*}
where the second inequality follows from the Sobolev embedding theorem.

For $L^1 \rightarrow L^2$ boundedness,  we analyze the kernel of the operator $F^{(l)} _M (\mathcal{L}, T) $. Using \Cref{reprsentation of kernel}, we can write
\begin{align*}
 &  P_{4B_j} \sum_{M=0} ^l \sum_{m= 1} ^{N_M} ( 1- \chi_{\widetilde{B_{j, m} ^M}}) F^{(l)} _M (\mathcal{L}, T) P_{\widetilde{B}_{j, m} ^M} f(x) \\
 & = P_{4B_j} \sum_{M=0} ^l \sum_{m= 1} ^{N_M} ( 1- \chi_{\widetilde{B_{j, m} ^M}}) \int_{\mathbb{R}^d} K_{F^{(l)}_M(\mathcal{L}, T)}(x, y) P_{\widetilde{B}_{j, m} ^M} f(y) \, dy.
\end{align*}
Consider the set
\begin{align*}
    B^y _j := \{x \in B (x_j, 2^{l+2} t): |x'' - y''| \geq  C 2^{\gamma l} 2^M t \cdot 2^l t   \}.
\end{align*}
Notice that  for any $x\in \supp ( 1- \chi_{\widetilde{B_{j, m} ^M}}) \chi_{B(x_j, 2^{l+2} t)} $ and  for any $y \in \supp P_{\widetilde{B}_{j, m} ^M} f$, we have $|x'' - x_{j, m} ^M| \geq C 2^{\gamma l +1} 2^M t \cdot 2^{l} t $ and $|y'' - x_{j, m} ^M| \leq C 2^M t \cdot 2^l t $ respectively. Therefore $|x'' - y''| \geq  C 2^{\gamma l} 2^M t \cdot 2^l t $, that is $x \in B^y _j$.

The above observation together with \Cref{Theorem: Weighted Plancherel in second variable} implies that
\begin{align}
\label{Estimate: L1 L2 kernel estimate in S_11}
    \Bigl( \int_{B^y _j} |K_{F^{(l)}_M(\mathcal{L}, T)}(x,y)|^2 \, dx \Bigr)^{1/2} & \leq C (2^{\gamma l}\, 2^M t \cdot 2^l t)^{-N}\, \Bigl( \int_{B^y _j} \left| |x'' - y''|^N\, K_{F^{(l)}_M(\mathcal{L}, T)}(x,y) \right|^2 \, dx \Bigr)^{1/2}\\
    & \nonumber \leq C (2^{\gamma l}\, 2^M t \cdot 2^l t)^{-N}\, 2^{M(N - d_2/2)}\, t^{2N}\, t^{-Q/2}\, \| \delta_{t^{-1}} (\psi F^{(l)}) (t\cdot)\|_{L^2_N}\\
    & \nonumber \leq C 2^{-\gamma l N} t^{-Q/2}\, \|F^{(l)}\|_{L^2}.
\end{align}
This estimate, along with Minkowski's integral inequality implies that
\begin{align*}
    & \left \| P_{4B_j} \sum_{M=0} ^l \sum_{m= 1} ^{N_M} ( 1- \chi_{\widetilde{B_{j, m} ^M}}) F^{(l)} _M (\mathcal{L}, T) P_{\widetilde{B}_{j, m} ^M} f    \right \|_{L^2} \\
    & \leq  \sum_{M=0} ^l \sum_{m= 1} ^{N_M} \int_{\widetilde{B}_{j, m} ^M} | f(y)| \, \Bigl( \int_{B^y _j} |K_{F^{(l)}_M(\mathcal{L}, T)}(x,y)|^2 \, dx \Bigr)^{1/2}\, dy\\
    &  \leq C  \sum_{M=0} ^l \sum_{m= 1} ^{N_M} \int_{\widetilde{B}_{j, m} ^M} | f(y)| \, 2^{-\gamma l N} t^{-Q/2}\, \|F^{(l)}\|_{L^2}\, dy\\
    & \leq C 2^{l\varepsilon} 2^{-\gamma l N} t^{-Q/2}\, \|F^{(l)}\|_{L^2}\, \|P_{\widetilde{B}_{j} } f \|_{L^1}.
\end{align*}
Interpolating between the $L^2 \rightarrow L^2$ and $L^1 \rightarrow L^2$ estimates obtained above yields,
\begin{align}
\label{Interpolation in error part for S11}
    &  \left \| P_{4B_j} \sum_{M=0} ^l \sum_{m= 1} ^{N_M} ( 1- \chi_{\widetilde{B_{j, m} ^M}}) F^{(l)} _M (\mathcal{L}, T) P_{\widetilde{B}_{j, m} ^M} f \right \|_{L^2} \\
   \nonumber &\leq C t^{-Q(1/p - 1/2)}\, 2^{-2\gamma l N(1/p - 1/2)}\, 2^{l\varepsilon} 2^{l (1 + d_2)(1-1/p)}\, \|F^{(l)}\|_{L^2}\, \|P_{\widetilde{B}_{j} } f \|_{L^p}.
\end{align}
By this estimate from (\ref{Inequality: Error part 1-chi for S_11}) and using (\ref{Inequality: BMO norm calculation}), we see that
\begin{align}
\label{Estimate: g_{j,<l} for S_{11} error part}
   & \left\| \sum_{M= 0} ^l \sum_{m= 1} ^{N_M}(1 - \chi_{\widetilde{B_{j, m} ^M}}) g_{j, m} ^M  \right\|_{L^q} \\
   & \nonumber \leq C (2^l t)^{Q(1/q - 1/2)}\, \|b\|_{BMO^{\varrho}(\mathbb{R}^d)} \, t^{-Q(1/p - 1/2)}\, 2^{-2\gamma l N(1/p - 1/2)}\,  \\
   & \nonumber \hspace{7cm} 2^{l\varepsilon} 2^{l (1 + d_2)(1-1/p)}\, \|F^{(l)}\|_{L^2}\, \|P_{\widetilde{B}_{j} } f \|_{L^p} \\
   & \nonumber \leq C (2^l t)^{Q(1/q - 1/2)}\, \|b\|_{BMO^{\varrho}(\mathbb{R}^d)} \, t^{-Q(1/p - 1/2)}\, 2^{-2\gamma l N(1/p - 1/2)}\, \\
   & \nonumber \hspace{5cm} 2^{l\varepsilon} 2^{l (1 + d_2)(1-1/p)}\, 2^{-l\beta} \|F\|_{W^2 _{\beta}}\, (2^l t)^{Q(1/p - 1/q)}\, \|P_{\widetilde{B}_{j} } f \|_{L^q}\\
   & \nonumber \leq C 2^{lQ(1/p - 1/2)}\, 2^{-2\gamma l N(1/p - 1/2)}\, 2^{l\varepsilon } 2^{l (1 + d_2)(1-1/p)}\, 2^{-l\beta} \|b\|_{BMO^{\varrho}(\mathbb{R}^d)} \|F\|_{W^2 _{\beta}}\, \|P_{\widetilde{B}_{j} } f \|_{L^q} \\
   & \nonumber \leq C 2^{-l\varepsilon} \|b\|_{BMO^{\varrho}(\mathbb{R}^d)} \|F\|_{W^2 _{\beta}} \|P_{\widetilde{B}_{j} } f \|_{L^q} ,
\end{align}
by choosing $N$ sufficiently large enough.

 As $\beta>d(1/p-1/2)$,  combining estimates obtained in (\ref{Estimate: g_{j,>l} for S_{11}}), (\ref{Estimate: g_{j, <l} for S_{11} main part}) and (\ref{Estimate: g_{j,<l} for S_{11} error part}), we conclude that
\begin{align*}
    S_{11} &\leq C 2^{-l\varepsilon} \|b\|_{BMO^{\varrho}(\mathbb{R}^d)} \|F\|_{W^2 _{\beta}}\, \|f \|_{L^q}.
\end{align*}

\underline{\emph{Estimation of $S_{12}$}:}
Next,  we turn our attention to $S_{12}$. Recall that 
\begin{align*}
 S_{12}=  C \left( \sum_j \|P_{4 B_j}\, (b - b_{B_j})\, ((1-\psi) F^{(l)}) (t\sqrt{\mathcal{L}}) P_{\widetilde{B}_j} f\|_{L^q}^q \right)^{1/q}.  
\end{align*}
We further decompose the function $1 -\psi$ as follows. Since $1-\psi$ is supported outside the interval $(1/8,2)$, we can choose a function $\phi\in C_c ^{\infty} (2,8)$ such that
\begin{align*}
    1 - \psi(\xi) = \sum_{\iota \geq 0} \phi(2^{-\iota} \xi) + \sum_{\iota \leq -6}  \phi(2^{-\iota} \xi) = \sum_{\iota \geq 0} \phi_{\iota} (\xi) + \sum_{\iota \leq -6} \phi_{\iota} (\xi), \quad \text{for all}\  \xi >0 .
\end{align*}
Therefore, by spectral theorem 
    \begin{equation}
        ((1 - \psi) F^{(l)})(t\sqrt{\mathcal{L}}) =\sum_{\iota \geq 0} (\phi_{\iota}  F^{(l)}) (t\sqrt{\mathcal{L}}) + \sum_{\iota \leq -6} (\phi_{\iota}  F^{(l)}) (t\sqrt{\mathcal{L}})\nonumber.
\end{equation} 
As before,  we consider the following two cases.

\underline{Case I:} Assume $|x_j'| \geq 2^{l+2} t$.

First,  using H\"older's inequality, we get
\begin{align*}
     & \|P_{4 B_j}\, (b - b_{B_j})\, ((1-\psi) F^{(l)})(t\sqrt{\mathcal{L}}) P_{\widetilde{B}_j} f\|_{L^q} \\
     & \leq \|P_{4 B_j}\, (b - b_{B_j}) \|_{L^{\frac{2q}{2-q}}} \| P_{4 B_j}\, ((1-\psi) F^{(l)})(t\sqrt{\mathcal{L}}) P_{\widetilde{B}_j} f \|_{L^2} .
\end{align*}
On the other hand,  from \Cref{Theorem: Restriction Estimate}, we see that
\begin{align*}
 & \| ((1- \psi) F^{(l)})(t\sqrt{\mathcal{L}}) P_{\widetilde{B}_j} f \|_{L^2}\\
 & \leq C \Bigl(\sum_{\iota \geq 0} + \sum_{\iota \leq -6}\Bigr)  \|  (\phi_{\iota} F^{(l)})(t\sqrt{\mathcal{L}}) P_{\widetilde{B}_j} f \|_{L^2} \\
 & \leq C \Bigl(\sum_{\iota \geq 0} + \sum_{\iota \leq -6}\Bigr) (2^{\iota} t^{-1})^{ (d_1 + d_2) (1/p - 1/2)}\, |x_j'|^{-d_2 (1/p -1/2)}\, \|\delta_{2^{\iota+3} t^{-1}}(\phi_{\iota}F^{(l)})(t\cdot)\|_{L^{\infty}} \|P_{\widetilde{B}_j} f\|_{L^p} .
\end{align*}
Notice that as $\supp{F} \subseteq [-1,1]$, $\supp{\phi} \subseteq [2,8] $ and $\check{\eta}$ is Schwartz class function, we have
\begin{align}
\label{Inequality: Calculation with cutoff phi}
 \|\phi_{\iota} F^{(l)}\|_{L^{\infty}} = 2^l \|\phi_{\iota} (F*\delta_{2^l} \check{\eta}) \|_{L^{\infty}} \leq C 2^{-N(l + \max{\{\iota,0}\})} \| F \|_{L^2} .
\end{align}
Using the above fact in the last expression, we further see that 
\begin{align*}
    & \| ((1- \psi) F^{(l)})(t\sqrt{\mathcal{L}}) P_{\widetilde{B}_j} f \|_{L^2}\\
    & \leq C \Bigl(\sum_{\iota \geq 0} + \sum_{\iota \leq -6}\Bigr) (2^{\iota} t^{-1})^{ (d_1 + d_2) (1/p - 1/2)}\, |x_j'|^{-d_2 (1/p -1/2)}\, 2^{-N(l + \max{\{\iota,0}\})} \| F \|_{L^2} \|P_{\widetilde{B}_j} f\|_{L^p} \\
    &\leq C t^{ -(d_1 + d_2) (1/p - 1/2)}\, |x_j'|^{-d_2 (1/p -1/2)}\, 2^{-Nl} \|  F \|_{L^2} \|P_{\widetilde{B}_j} f\|_{L^p} .
\end{align*}
Therefore, using the above estimate, (\ref{Inequality: BMO norm calculation}) and applying H\"older's inequality,  we obtain that
\begin{align*} 
     &\| P_{4 B_j}\, (b - b_{B_j})\, ((1- \psi) F^{(l)})(t\sqrt{\mathcal{L}}) P_{\widetilde{B}_j} f \|_{L^q}\\
     & \leq C \|b\|_{BMO^{\varrho}(\mathbb{R}^d)} |B(x_j, 2^{l+2} t)|^{(1/q - 1/2)}\,  \|  ((1-\psi) F^{(l)}) (t\sqrt{\mathcal{L}}) P_{\widetilde{B}_j} f \|_{L^2}\\
     & \leq C \|b\|_{BMO^{\varrho}(\mathbb{R}^d)}\, (2^l t)^{(d_1 + d_2) (1/q - 1/2)} |x_j'|^{d_2(1/q - 1/2)}  \,\\
     & \hspace{5cm} t^{ -(d_1 + d_2) (1/p - 1/2)}\, |x_j'|^{-d_2 (1/p -1/2)}\, 2^{-Nl} \|  F \|_{L^2} \|P_{\widetilde{B}_j} f\|_{L^p} \\
     & \leq C \|b\|_{BMO^{\varrho}(\mathbb{R}^d)}\, (2^l t)^{(d_1 + d_2) (1/q - 1/2)} |x_j'|^{d_2(1/q - 1/2)} \, t^{ -(d_1 + d_2) (1/p - 1/2)}\, |x_j'|^{-d_2 (1/p -1/2)}  \\
     & \hspace{5cm} 2^{-Nl}\|  F \|_{L^2}\, (2^l t)^{(d_1 + d_2) (1/p - 1/q)} |x_j'|^{d_2(1/p - 1/q)} \|P_{\widetilde{B}_j} f\|_{L^q}\\
     & \leq C \|b\|_{BMO^{\varrho}(\mathbb{R}^d)}\, 2^{l(d_1 +d _2) (1/p - 1/2)}\, 2^{-Nl} \|  F \|_{L^2}\, \|P_{\widetilde{B}_j} f\|_{L^q} .
\end{align*}
So that choosing $N$ large enough, we get
\begin{align*}
    S_{12} &\leq C 2^{-l\varepsilon} \|b\|_{BMO^{\varrho}(\mathbb{R}^d)} \|  F \|_{L^2}\, \| f\|_{L^q} .
\end{align*}

\underline{Case II:} Assume $|x_j'| < 2^{l+2} t$.

Again, let $\Theta$ be defined as in (\ref{Definition: Cutoff function chi}). Define the function $F^{(l,\iota)}_M$ as
\begin{align*}
    F^{(l,\iota)}_M (\kappa, r) &= (\phi_{\iota} F^{(l)}) (t\sqrt{\kappa}) \,\Theta_{M}(\kappa/r)  \quad \text{for} \quad \kappa \geq 0,\ r \neq 0
\end{align*}
and $F^{(l,\iota)}_M (\kappa, r) = 0$ else. This allows us to write
\begin{align*}
    & P_{4 B_j}\, (b - b_{B_j})\, ((1-\psi) F^{(l)})(t\sqrt{\mathcal{L}}) P_{\widetilde{B}_j} f \\
    &=  P_{4 B_j}\, (b - b_{B_j})\, \Bigl(\sum_{\iota \geq 0} + \sum_{\iota \leq -6}\Bigr) \, \left( \sum_{M=0 } ^l + \sum_{M=l+ 1 } ^{\infty}\right) F^{(l,\iota)}_M (\mathcal{L}, T) P_{\widetilde{B}_j} f \\
    &=:\Bigl(\sum_{\iota \geq 0} + \sum_{\iota \leq -6}\Bigr) g_{j, \leq l} ^{\iota} + \Bigl(\sum_{\iota \geq 0} + \sum_{\iota \leq -6}\Bigr)g_{j, > l} ^{\iota} .
\end{align*}
For the second term, first note that from \Cref{Theorem: Restriction Estimate} and (\ref{Inequality: Calculation with cutoff phi}), we have
\begin{align}
\label{Estimate of S_{12} case II second term}
    & \left\|P_{4 B_j} \Bigl(\sum_{\iota \geq 0} + \sum_{\iota \leq -6}\Bigr) \,  \sum_{M=l+ 1 } ^{\infty} F^{(l,\iota)}_M (\mathcal{L}, T) P_{\widetilde{B}_j} f  \right\|_{L^2} \\
    &\nonumber \leq \Bigl(\sum_{\iota\geq 0} + \sum_{\iota \leq -6}\Bigr) \left \|  \sum_{M=l+ 1 } ^{\infty} F^{(l,\iota)}_M (\mathcal{L}, T) P_{\widetilde{B}_j} f  \right\|_{L^2} \\
    &\nonumber \leq C \Bigl(\sum_{\iota \geq 0} + \sum_{\iota \leq -6}\Bigr) (2^{\iota} t^{-1})^{Q(1/p - 1/2)}\, 2^{-ld_2 (1/p - 1/2)}\, \|\delta_{2^{\iota+3} t^{-1}}(\phi_{\iota} F^{(l)})(t\cdot)\|_{L^{\infty}} \, \|P_{\widetilde{B}_j} f\|_{L^p} \\
    &\nonumber \leq C \Bigl(\sum_{\iota \geq 0} + \sum_{\iota \leq -6}\Bigr) (2^{\iota} t^{-1})^{Q(1/p - 1/2)}\, 2^{-ld_2 (1/p - 1/2)}\, 2^{-N(l + \max{\{\iota,0}\})} \|  F \|_{2} \, \|P_{\widetilde{B}_j} f\|_{L^p} \\
    &\nonumber \leq C t^{-Q(1/p - 1/2)}\, 2^{-ld_2 (1/p - 1/2)}\, 2^{-Nl}\|F\|_{L^{2}} \, \|P_{\widetilde{B}_j} f\|_{L^p} .
\end{align}
Therefore using H\"older's inequality, the above estimate, (\ref{Inequality: BMO norm calculation}), again H\"older's inequality and choosing $N$ large enough, we find that 
\begin{align}
\label{Estimate: g_{j,>l} for S_{12}}
   & \left\|\Bigl(\sum_{\iota \geq 0} + \sum_{\iota \leq -6}\Bigr)g_{j, > l} ^{\iota} \right\|_{L^q}\\
   & \nonumber \leq  \|P_{4 B_j}\, (b - b_{B_j})\|_{L^{\frac{2q}{2-q}}}  \left\| P_{4 B_j}\, \, \Bigl(\sum_{\iota \geq 0} + \sum_{\iota \leq -6}\Bigr) \,  \sum_{M=l+ 1 } ^{\infty} F^{(l,\iota)}_M (\mathcal{L}, T) P_{\widetilde{B}_j} f  \right\|_{L^2} \\
   & \nonumber \leq C \|b\|_{BMO^{\varrho}(\mathbb{R}^d)} (2^l t) ^{Q(1/q - 1/2)}\,  t^{-Q(1/p - 1/2)}\, 2^{-ld_2 (1/p - 1/2)}\, 2^{-N l} \|F\|_{L^2} \, \|P_{\widetilde{B}_j} f\|_{L^p} \\
   & \nonumber \leq C \|b\|_{BMO^{\varrho}(\mathbb{R}^d)} (2^l t) ^{Q(1/q - 1/2)}\,  t^{-Q(1/p - 1/2)}\, 2^{-ld_2 (1/p - 1/2)}\, 2^{-N l} \|F\|_{L^2} (2^l t)^{Q(1/p- 1/q)} \|P_{\widetilde{B}_j} f\|_{L^q}\\
   & \nonumber \leq C 2^{ld(1/p - 1/2)}\, 2^{-N l} \|b\|_{BMO^{\varrho}(\mathbb{R}^d)}  \|F\|_{L^2} \, \|P_{\widetilde{B}_j} f\|_{L^q}\\
   & \nonumber \leq C 2^{-l\varepsilon} \|b\|_{BMO^{\varrho}(\mathbb{R}^d)} \|F\|_{L^2} \, \|P_{\widetilde{B}_j} f\|_{L^q}.
\end{align}

For the first term $\Bigl(\sum_{\iota \geq 0} + \sum_{\iota \leq -6}\Bigr) g_{j, \leq l} ^{\iota}$, we utilize the decomposition applied in the estimation of $S_{11}$ for the term $g_{j, \leq l}$ of $ \widetilde{B}_j $ and $P_{\widetilde{B}_j} f$. With the same notation, let us set 
\begin{align*}
    g_{j, m} ^{\iota,M} = P_{4 B_j}\, (b - b_{B_j}) \, F^{(l,\iota)}_M (\mathcal{L}, T) P_{\widetilde{B}_{j, m} ^M} f ,
\end{align*}
and write 
\begin{align*}
   \Bigl(\sum_{\iota \geq 0} + \sum_{\iota \leq -6}\Bigr) g_{j, \leq l} ^{\iota} &= \Bigl(\sum_{\iota \geq 0} + \sum_{\iota \leq -6}\Bigr) \sum_{M= 0} ^l \sum_{m= 1} ^{N_M} \chi_{\widetilde{B_{j, m} ^M}} g_{j, m} ^{\iota,M} + \Bigl(\sum_{\iota \geq 0} + \sum_{\iota \leq -6}\Bigr) \sum_{M= 0} ^l \sum_{m= 1} ^{N_M}(1 - \chi_{\widetilde{B_{j, m} ^M}}) g_{j, m} ^{\iota,M} .
\end{align*}

From H\"older's inequality and the bounded overlapping property of the balls $\widetilde{B_{j, m} ^M}$, we see that
\begin{align*}
  \left\| \Bigl(\sum_{\iota \geq 0} + \sum_{\iota \leq -6}\Bigr)  \sum_{M= 0} ^l \sum_{m= 1} ^{N_M} \chi_{\widetilde{B_{j, m} ^M}} g_{j, m} ^{\iota, M} \right\|_{L^q} ^q & \leq C ((l+ 1)N_{\gamma})^{q-1} \sum_{M= 0} ^l \sum_{m= 1} ^{N_M} \left\| \Bigl(\sum_{\iota \geq 0} + \sum_{\iota \leq -6}\Bigr) \chi_{\widetilde{B_{j, m} ^M}} g_{j, m} ^{\iota,M} \right\|_{L^q} ^q .
\end{align*}

Now using H\"older's inequality together with (\ref{Inequality: BMO norm calculation}), \Cref{Theorem: Restriction Estimate} and (\ref{Inequality: Calculation with cutoff phi}) and again H\"older's inequality, we find that
\begin{align*}
    & \left\| \Bigl(\sum_{\iota \geq 0} + \sum_{\iota \leq -6}\Bigr) \chi_{\widetilde{B_{j, m} ^M}} g_{j, m} ^{\iota,M} \right\|_{L^q}\\
    & \leq \left\|  P_{4 B_j}\, (b - b_{B_j})  \right\|_{L^{\frac{sq}{s-q}}}\, \left\| \chi_{\widetilde{B_{j, m} ^M}} \Bigl(\sum_{\iota \geq 0} + \sum_{\iota \leq -6}\Bigr) F^{(l,\iota)}_M (\mathcal{L}, T) P_{\widetilde{B}_{j, m} ^M} f \right\|_{L^s}\\
    & \leq C (2^l t)^{Q(1/q - 1/s)} \, \|b\|_{BMO^{\varrho}(\mathbb{R}^d)} (2^l t)^{d_1(1/s - 1/2)}\, (2^{\gamma l} 2^M t \cdot 2^l t)^{d_2(1/s - 1/2)} \\
    & \hspace{8cm} \left\| \Bigl(\sum_{\iota \geq 0} + \sum_{\iota \leq -6}\Bigr) F^{(l,\iota)}_M (\mathcal{L}, T) P_{\widetilde{B}_{j, m} ^M} f \right\|_{L^2}\\
    & \leq C (2^l t)^{Q(1/q - 1/s)} \, \|b\|_{BMO^{\varrho}(\mathbb{R}^d)} (2^l t)^{d_1(1/s - 1/2)}\, (2^{\gamma l} 2^M t \cdot 2^l t)^{d_2(1/s - 1/2)} \Bigl(\sum_{\iota \geq 0} + \sum_{\iota \leq -6}\Bigr)  \\
    & \hspace{3cm} (2^{\iota} t^{-1})^{Q(1/p - 1/2)} 2^{-Md_2(1/p - 1/2)}\, \|\delta_{2^{\iota+3} t^{-1}}(\phi_{\iota} F^{(l)})(t\cdot)\|_{L^{\infty}} \,\, \|P_{\widetilde{B}_{j, m} ^M} f\|_{L^p} \\
    & \leq C (2^l t)^{Q(1/q - 1/s)} \, \|b\|_{BMO^{\varrho}(\mathbb{R}^d)} (2^l t)^{d_1(1/s - 1/2)}\, (2^{\gamma l} 2^M t \cdot 2^l t)^{d_2(1/s - 1/2)} 2^{-Md_2(1/p - 1/2)} \\
    & \hspace{4cm} \Bigl(\sum_{\iota \geq 0} + \sum_{\iota \leq -6}\Bigr)  (2^{\iota} t^{-1})^{Q(1/p - 1/2)} 2^{-N(l + \max{\{\iota,0}\})} \|  F \|_{L^2} \\
    & \hspace{6cm} (2^l t)^{d_1(1/p - 1/q)} (2^M t\cdot  2^l t)^{d_2(1/p - 1/q)}\, \|P_{\widetilde{B}_{j, m} ^M} f\|_{L^q} \\
    & \leq C (2^l t)^{Q(1/q - 1/s)} \, \|b\|_{BMO^{\varrho}(\mathbb{R}^d)} (2^l t)^{d_1(1/s - 1/2)}\, (2^{\gamma l} 2^M t \cdot 2^l t)^{d_2(1/s - 1/2)}  t^{-Q(1/p - 1/2)} \\
    & \hspace{3cm} 2^{-Md_2(1/p - 1/2)}\, 2^{-Nl} \|  F \|_{L^2}\, (2^l t)^{d_1(1/p - 1/q)} (2^M t\cdot  2^l t)^{d_2(1/p - 1/q)}\, \|P_{\widetilde{B}_{j, m} ^M} f\|_{L^q} \\
    &\leq C 2^{ ld(1/p - 1/2) }\, 2^{l\varepsilon}\,2^{-Nl}  \|b\|_{BMO^{\varrho}(\mathbb{R}^d)} \|  F \|_{L^2} \|P_{\widetilde{B}_{j, m} ^M} f\|_{L^q} ,
\end{align*}
for some $\varepsilon>0$, by choosing $s$ very close to $q$ and $\gamma>0$ small enough before.
 
Hence, 
\begin{align}
\label{Estimate: g_{j, <l} for S_{12} main part}
   & \left\| \Bigl(\sum_{i \geq 0} + \sum_{i \leq -6}\Bigr)  \sum_{M= 0} ^l \sum_{m= 1} ^{N_M} \chi_{\widetilde{B_{j, m} ^M}} g_{j, m} ^{i, M} \right\|_{L^q} ^q\\
   & \nonumber \leq C ((l+ 1)N_{\gamma})^{q-1} \sum_{M= 0} ^l \sum_{m= 1} ^{N_M} 2^{ ldq(1/p - 1/2) }\, 2^{lq\varepsilon}\,2^{-Nlq} \, \|b\|_{BMO^{\varrho}(\mathbb{R}^d)}^q \|  F \|_{L^2}^q \|P_{\widetilde{B}_{j, m} ^M} f\|_{L^q}^q \\
   & \nonumber \leq C 2^{-lq\varepsilon}  \|b\|_{BMO^{\varrho}(\mathbb{R}^d)} ^q\, \|F\|_{L^2} ^q\,\|P_{\widetilde{B}_j} f\|_{L^q} ^q,
\end{align}
by choosing $N$ sufficiently large.

Estimate for the remainder term $\displaystyle{\Bigl(\sum_{\iota \geq 0} + \sum_{\iota \leq -6}\Bigr) \sum_{M= 0} ^l \sum_{m= 1} ^{N_M}(1 - \chi_{\widetilde{B_{j, m} ^M}}) g_{j, m} ^{\iota,M}}$. In this case, we use argument similar to that used in the estimation of $S_{11}$ for the term $\displaystyle{\sum_{M= 0} ^l \sum_{m= 1} ^{N_M}(1 - \chi_{\widetilde{B_{j, m} ^M}}) g_{j, m} ^M }$. Let $K_{F^{(l,\iota)}_M (\mathcal{L}, T)}$ denote the kernel associated to the operator $F^{(l,\iota)}_M (\mathcal{L}, T) $. Applying \Cref{Theorem: Weighted Plancherel in second variable} for $N=0$, we get
\begin{align*}
    \Bigl( \int_{\mathbb{R}^d} |K_{F^{(l,\iota)}_M (\mathcal{L}, T)}(x,y)|^2 \, dx \Bigr)^{1/2} & \leq C \, 2^{-M d_2/2}\, (2^{\iota} t^{-1})^{Q/2}\, \| \delta_{2^{\iota +3}t^{-1}} (\phi_{\iota} F^{(l)})(t\cdot)\|_{L^{\infty}} \\
    & \leq  C \, 2^{\iota Q/2}  \, t^{-Q/2}\, \| \phi_{\iota} F^{(l)} \|_{L^{\infty}} .
\end{align*}

This estimate and an application of Minkowski's inequality, implies that 
\begin{align*}
    \left \| \sum_{M= 0} ^l \sum_{m= 1} ^{N_M}(1 - \chi_{\widetilde{B_{j, m} ^M}})  F^{(l,\iota)} _M (\mathcal{L}, T) P_{\widetilde{B}_{j, m} ^M} f \right \|_{L^2} & \leq  \,C \sum_{M=0} ^l \, 2^{\iota Q/2}  \, t^{-Q/2}\, \| \phi_{\iota} F^{(l)} \|_{L^{\infty}}\, \|P_{\widetilde{B}_{j} } f \|_{L^1}\\
     & \leq C  2^{l \varepsilon} \, 2^{\iota Q/2}  \, t^{-Q/2}\,  \| \phi_{\iota} F^{(l)} \|_{L^{\infty}}\, \|P_{\widetilde{B}_{j} } f \|_{L^1} .
\end{align*}
Also, using Plancherel theorem, H\"older's inequality and upper bound of $N_M$ we see that
\begin{align*}
\left \| \sum_{M= 0} ^l \sum_{m= 1} ^{N_M}(1 - \chi_{\widetilde{B_{j, m} ^M}})  F^{(l,\iota)}_M (\mathcal{L}, T) P_{\widetilde{B}_{j, m} ^M} f \right \|_{L^2} & \leq C \|\phi_{\iota} F^{(l)} \|_{L^{\infty}} \sum_{M= 0} ^l \sum_{m = 1} ^{N_M}\| P_{\widetilde{B}_{j, m} ^M} f \|_{L^2}\\
& \leq C \|\phi_{\iota} F^{(l)} \|_{L^{\infty}} \, 2^{l \varepsilon} 2^{l d_2/2} \, \|P_{\widetilde{B}_{j} } f \|_{L^2} .
\end{align*}
Hence, interpolating above results we get
\begin{align}
 \label{interpolation in remaindeder term for S_{12}}   &  \left \| \sum_{M= 0} ^l \sum_{m= 1} ^{N_M}(1 - \chi_{\widetilde{B_{j, m} ^M}})  F^{(l,\iota)}_M (\mathcal{L}, T) P_{\widetilde{B}_{j, m} ^M} f \right \|_{L^2}\\
    \nonumber&\leq C 2^{l\varepsilon }\, t^{-Q(1/p - 1/2)}\, 2^{\iota Q (1/p-1/2)} \, 2^{l d_2 (1-1/p)}\, \|\phi_{\iota} F^{(l)}\|_{L^{\infty}}\, \|P_{\widetilde{B}_{j} } f \|_{L^p}
\end{align}
for all $1\leq p < 2$.

Now using above estimate together with H\"older's inequality and (\ref{Inequality: BMO norm calculation}), (\ref{Inequality: Calculation with cutoff phi}) and again H\"older's inequality, we get
\begin{align}
\label{Inequality: Error part 1-chi for S_12}
   & \left \| P_{4 B_j}\, (b - b_{B_j}) \, \Bigl(\sum_{\iota \geq 0} + \sum_{\iota \leq -6}\Bigr) \sum_{M= 0} ^l \sum_{m= 1} ^{N_M}(1 - \chi_{\widetilde{B_{j, m} ^M}})  F^{(l,\iota)}_M (\mathcal{L}, T) P_{\widetilde{B}_{j, m} ^M} f  \right \|_{L^q} \\
   & \nonumber \leq C (2^l t)^{Q(1/q - 1/2)} \|b\|_{BMO^{\varrho}(\mathbb{R}^d)}\, \Bigl(\sum_{\iota \geq 0} + \sum_{\iota \leq -6}\Bigr)\left \| \sum_{M= 0} ^l \sum_{m= 1} ^{N_M}(1 - \chi_{\widetilde{B_{j, m} ^M}})  F^{(l,\iota)}_M (\mathcal{L}, T) P_{\widetilde{B}_{j, m} ^M} f  \right \|_{L^2}\\
    & \nonumber \leq C (2^l t)^{Q(1/q - 1/2)} \|b\|_{BMO^{\varrho}(\mathbb{R}^d)}\,\Bigl(\sum_{\iota \geq 0} + \sum_{\iota \leq -6}\Bigr) 2^{l\varepsilon }\,  t^{-Q(1/p - 1/2)}\, 2^{\iota Q (1/p-1/2)} \, 2^{l d_2 (1-1/p)} \\
    & \nonumber\hspace{9cm} 2^{-N(l + \max{\{\iota,0}\})} \,  \|F\|_{L^2}\, \|P_{\widetilde{B}_{j} } f \|_{L^p}\\
    & \nonumber \leq C (2^l t)^{Q(1/q - 1/2)}\, \|b\|_{BMO^{\varrho}(\mathbb{R}^d)} \, 2^{l\varepsilon }\,  t^{-Q(1/p - 1/2)}\,  2^{-Nl } \, 2^{l d_2 (1-1/p)}\, \\
   & \nonumber \hspace{8cm}  \|F\|_{L^2}\, (2^l t)^{Q(1/p - 1/q)} \|P_{\widetilde{B}_{j} } f \|_{L^q}\\
   & \nonumber \leq C 2^{lQ(1/p - 1/2)}\,2^{l\varepsilon }\, 2^{-Nl } 2^{l d_2 (1-1/p)}\, \|b\|_{BMO^{\varrho}(\mathbb{R}^d)} \|F\|_{L^2}\, \|P_{\widetilde{B}_{j} } f \|_{L^q} \\
   & \nonumber \leq C 2^{-l\varepsilon }\,  \|b\|_{BMO^{\varrho}(\mathbb{R}^d)} \|F\|_{L^2}\, \|P_{\widetilde{B}_{j} } f \|_{L^q},
\end{align}
by choosing $N$ large enough.

Finally, combining the estimates (\ref{Estimate: g_{j,>l} for S_{12}}), (\ref{Estimate: g_{j, <l} for S_{12} main part}) and (\ref{Inequality: Error part 1-chi for S_12}),  we conclude that
\begin{align*}
    S_{12} &\leq C 2^{-l\varepsilon} \|b\|_{BMO^{\varrho}(\mathbb{R}^d)} \|F\|_{L^2}\, \|f \|_{L^q}.
\end{align*}

\underline{\emph{Estimation of $S_2$}:}

In this situation, we modified slightly the argument  used in the estimation of $S_1$. So, for the shake of brevity, we will be concise as much as possible. Let $\psi$ be as in $S_1$. Then, we have that
\begin{align*}
    S_2^{1/q} & \leq C \left( \sum_j \|P_{4 B_j}\,  (\psi F^{(l)})(t\sqrt{\mathcal{L}}) P_{\widetilde{B}_j} (b - b_{B_j}) f\|_{L^q}^q \right)^{1/q} \\
    & \hspace{5cm} + \left( \sum_j \|P_{4 B_j}\, ((1-\psi) F^{(l)}) (t\sqrt{\mathcal{L}}) P_{\widetilde{B}_j} (b - b_{B_j})f\|_{L^q}^q \right)^{1/q} \\
    & := S_{21} + S_{22}.
\end{align*}

\underline{\emph{Estimation of $S_{21}$}:}
As earlier, we consider the following two cases.

\underline{Case I:} Assume $|x_j'| \geq 2^{l+2} t$.

Applying H\"older's inequality, \Cref{Theorem: Restriction Estimate}, again H\"older's inequality and (\ref{Inequality: BMO norm calculation}), we have
\begin{align*}    
     & \|P_{4 B_j}\, (\psi F^{(l)})(t\sqrt{\mathcal{L}}) P_{\widetilde{B}_j} (b - b_{B_j})f\|_{L^q} \\
     & \leq C  |B(x_j, 2^{l+2} t)|^{1/q - 1/2}\,  \|  (\psi F^{(l)}) (t\sqrt{\mathcal{L}}) P_{\widetilde{B}_j} \, (b - b_{B_j}) f \|_{L^2}\\
     & \leq C  (2^l t)^{(d_1 + d_2) (1/q - 1/2)} |x_j'|^{d_2(1/q - 1/2)} \, t^{- (d_1 + d_2) (1/p - 1/2)}\, |x_j'|^{-d_2 (1/p -1/2)}\, \\
     & \hspace{7cm} \|\delta_{t^{-1}}(\psi F^{(l)})(t \cdot)\|_{L^2} \|P_{\widetilde{B}_j} \, (b - b_{B_j}) f\|_{L^p}\\
     & \leq C (2^l t)^{(d_1 + d_2) (1/q - 1/2)} |x_j'|^{d_2(1/q - 1/2)} \, t^{- (d_1 + d_2) (1/p - 1/2)}\, |x_j'|^{-d_2 (1/p -1/2)}\, \|\psi F^{(l)}\|_2 \, \\
     & \hspace{7cm}\|P_{\widetilde{B}_j}(b- b_{B_j})\|_{L^{\frac{qp}{q-p}}}\, \|P_{\widetilde{B}_j} \,  f\|_{L^q} \\
    & \leq C 2^{l(d_1 + d_2) (1/p - 1/2)} \, 2^{-\beta l}\, \|b\|_{BMO^{\varrho}(\mathbb{R}^d)} \| F\|_{W^2 _{\beta}}  \|P_{\widetilde{B}_j} \,  f\|_{L^q} .
\end{align*}

Therefore as $\beta>d(1/p-1/2)$, we get
\begin{align*}
    S_{21} &\leq C 2^{- l \varepsilon} \|b\|_{BMO^{\varrho}(\mathbb{R}^d)} \|F\|_{W^2_{\beta}} \|f\|_{L^q} .
\end{align*}

\underline{Case II:} Assume $|x_j'| < 2^{l+2} t$.

As in case II of estimation of $S_{11}$, we  write 
\begin{align*}
    P_{4 B_j}\, (\psi F^{(l)})(t\sqrt{\mathcal{L}}) P_{\widetilde{B}_j}\, (b - b_{B_j}) f &=  P_{4 B_j} \, \left( \sum_{M=0 } ^l + \sum_{M=l+ 1 } ^{\infty}\right) F^{(l)}_M (\mathcal{L}, T) P_{\widetilde{B}_j} (b - b_{B_j})f \\
    &=: \Tilde{g}_{j, \leq l}  + \Tilde{g}_{j, > l} .
\end{align*}

By an argument similar to that used in the estimate for the term $g_{j,>l}$  of $S_{11}$, using H\"older's inequality, \Cref{Theorem: Restriction Estimate}, (\ref{Inequality: BMO norm calculation}) we conclude that
\begin{align}
\label{Estimate: g_{j,>l} for S_{21}}
   & \| \Tilde{g}_{j, > l}\|_{L^q}\\
   & \nonumber  \leq C  (2^l t) ^{Q(1/q - 1/2)}\, t^{-Q(1/p - 1/2)}\, 2^{-ld_2 (1/p - 1/2)}\, \|F^{(l)}\|_{L^2} \\
   & \hspace{7cm} \nonumber (2^l t) ^{Q(1/p - 1/q)}\, \|b\|_{BMO^{\varrho}(\mathbb{R}^d)}  \|P_{\widetilde{B}_j} f\|_{L^q} \\
   & \nonumber  \leq C \|b\|_{BMO^{\varrho}(\mathbb{R}^d)} 2^{ld(1/p - 1/2)}\,2^{-\beta l} \|F\|_{W^2_{\beta}} \, \|P_{\widetilde{B}_j} f\|_{L^q} .
\end{align}

Next, we settle the term $\Tilde{g}_{j, \leq l}$. Let $\widetilde{B}_{j, m} ^M$, $\chi_{\widetilde{B_{j, m} ^M}}$ and the decomposition of $f$ be the same as in Case II of estimation of $S_{11}$. Also set $\Tilde{g}_{j, m} ^M = P_{4 B_j}\, F^{(l)}_M (\mathcal{L}, T) P_{\widetilde{B}_{j, m} ^M} (b - b_{B_j})  f $ and write 
\begin{align*}
    \Tilde{g}_{j, \leq l} &= \sum_{M= 0} ^l \sum_{m= 1} ^{N_M} \chi_{\widetilde{B_{j, m} ^M}} \Tilde{g}_{j, m} ^M + \sum_{M= 0} ^l \sum_{m= 1} ^{N_M}(1 - \chi_{\widetilde{B_{j, m} ^M}}) \Tilde{g}_{j, m} ^M .
\end{align*}

For the first term, choose $r$ such that $p < r < q < 2$. Then using H\"older's inequality together with \Cref{Theorem: Restriction Estimate}, (\ref{Inequality: BMO norm calculation}) and H\"older's inequality,  we obtain that
\begin{align*}
    & \left\|  \chi_{\widetilde{B_{j, m} ^M}} \Tilde{g}_{j, m} ^M \right\|_{L^q} = \left\| \chi_{\widetilde{B_{j, m} ^M}} F^{(l)}_M (\mathcal{L}, T) P_{\widetilde{B}_{j, m} ^M} (b - b_{B_j}) f \right\|_{L^q} \\
   & \leq C |\widetilde{B_{j, m} ^M}|^{1/q - 1/2}  \left \|   F^{(l)}_M (\mathcal{L}, T) P_{\widetilde{B}_{j, m} ^M} (b - b_{B_j}) f \right\|_{L^2} \\
   &\leq C  (2^l t)^{d_1 (1/q - 1/2)}\, (2^{\gamma l} 2^M t\cdot 2^l t)^{d_2(1/q - 1/2)}\,t^{-Q(1/p - 1/2)}\, 2^{-M d_2 (1/p - 1/2)}\,\\
   &\hspace{6cm} \|\delta_{t^{-1}}(\psi F^{(l)})(t \cdot)\|_{L^2}\, \|P_{\widetilde{B}_{j, m} ^M} (b - b_{B_j}) f\|_{L^p} \\
   &\leq C  (2^l t)^{d_1 (1/q - 1/2)}\, (2^{\gamma l} 2^M t\cdot 2^l t)^{d_2(1/q - 1/2)}\,t^{-Q(1/p - 1/2)}\, 2^{-M d_2 (1/p - 1/2)}\,\\
   &\hspace{6cm} \|\psi F^{(l)}\|_{L^2}\, \|P_{4B_j} (b - b_{B_j}) \|_{L^{\frac{rp}{r-p}}}  \, \|P_{\widetilde{B}_{j, m} ^M} f\|_{L^r}  \\
   & \leq C (2^l t)^{d_1 (1/q - 1/2)}\, (2^{\gamma l} 2^M t\cdot 2^l t)^{d_2(1/q - 1/2)}\,t^{-Q(1/p - 1/2)}\, 2^{-M d_2 (1/p - 1/2)}\,\\
   &\hspace{2cm}\|F^{(l)}\|_{L^2}\, (2^l t )^{Q(1/p -1/r)}\, \|b\|_{BMO^{\varrho}(\mathbb{R}^d)} (2^l t)^{d_1(1/r - 1/q)}\, (2^Mt\cdot 2^l t)^{d_2(1/r - 1/q)} \,\|P_{\widetilde{B}_{j, m} ^M} f\|_{L^q}\\
   & \leq C 2^{l\varepsilon} 2^{lQ(1/p-1/r) + ld( 1/r - 1/2)}\, \|F^{(l)}\|_{L^2}\,  2^{-Md_2(1/p - 1/r)} \|b\|_{BMO^{\varrho}(\mathbb{R}^d)}\, \|P_{\widetilde{B}_{j, m} ^M} f\|_{L^q}\\
   & \leq  C 2^{l\varepsilon + ld( 1/p - 1/2)}\, 2^{-\beta l}\, \|F\|_{W^2_{\beta}} \|b\|_{BMO^{\varrho}(\mathbb{R}^d)}\,  \|P_{\widetilde{B}_{j, m} ^M} f\|_{L^q},
\end{align*}
for some $\varepsilon>0$, by choosing $r$ very close to $p$ and $\gamma>0$ small enough.

As in the Case II of estimation of $S_{11}$ from the above estimate together with H\"older's inequality and bounded overlapping property of the balls $\widetilde{B_{j, m} ^M}$, we obtain
\begin{align}
\label{Estimate: g_{j, <l} for S_{21} main part}
  \left\|  \sum_{M= 0} ^l \sum_{m= 1} ^{N_M} \chi_{\widetilde{B_{j, m} ^M}} \Tilde{g}_{j, m} ^M \right\|_{L^q} ^q &\leq C 2^{lq\varepsilon} 2^{ ldq (1/p - 1/2) - l \beta q}\, \|b\|_{BMO^{\varrho}(\mathbb{R}^d)} ^q\, \|F\|_{W_{\beta} ^2} ^q\, \|P_{\widetilde{B}_j} f\|_{L^q} ^q .
\end{align}

We now proceed with the estimation of the remainder term $\sum_{M= 0} ^l \sum_{m= 1} ^{N_M}(1 - \chi_{\widetilde{B_{j, m} ^M}}) \Tilde{g}_{j, m} ^M$. By arguments similar to those used in proving (\ref{Interpolation in error part for S11}), we get
\begin{align*}
    &  \left \| P_{4B_j} \sum_{M=0} ^l \sum_{m= 1} ^{N_M} ( 1- \chi_{\widetilde{B_{j, m} ^M}}) F^{(l)} _M (\mathcal{L}, T) P_{\widetilde{B}_{j, m} ^M} (b- b_{B_j})  f    \right \|_{L^2}\\
    & \leq C t^{-Q(1/p - 1/2)}\, 2^{-2\gamma l N(1/p - 1/2)}\, 2^{l\varepsilon} 2^{l (1 + d_2)(1-1/p)}\, \|F^{(l)}\|_{L^2}\, \|P_{\widetilde{B}_{j} } (b- b_{B_j})  f \|_{L^p} .
\end{align*}

Above estimates together with H\"older's inequality and (\ref{Inequality: BMO norm calculation}),  we conclude that
\begin{align}
\label{Estimate: g_{j,<l} for S_{21} error part}
   & \left \| P_{4B_j} \sum_{M= 0} ^l \sum_{m= 1} ^{N_M}(1 - \chi_{\widetilde{B_{j, m} ^M}}) \Tilde{g}_{j, m} ^M \right \|_{L^q}\\
   & \nonumber \leq C (2^l t)^{Q(1/q - 1/2)}\,  t^{-Q(1/p - 1/2)}\, 2^{-2\gamma l N(1/p - 1/2)}\, 2^{l\varepsilon} 2^{l (1 + d_2)(1-1/p)}\, \|F^{(l)}\|_{L^2}\, \|P_{\widetilde{B}_{j} } (b- b_{B_j})  f \|_{L^p} \\
   & \nonumber \leq C (2^l t)^{Q(1/q - 1/2)}\,  t^{-Q(1/p - 1/2)}\, 2^{-2\gamma l N(1/p - 1/2)}\, 2^{l\varepsilon} 2^{l (1 + d_2)(1-1/p)}\, 2^{-l\beta} \|F\|_{W^2 _{\beta}}\, \\
    & \nonumber \hspace{9cm} \,(2^l t)^{Q(1/p - 1/q)}\, \|b\|_{BMO^{\varrho}(\mathbb{R}^d)} \|P_{\widetilde{B}_j}f\|_{L^q}\\
    & \nonumber \leq C  2^{-l \varepsilon}\,\|F\|_{W^2 _{\beta}}\, \|b\|_{BMO^{\varrho}(\mathbb{R}^d)} \,\|P_{\widetilde{B}_j}f\|_{L^q},
\end{align}
by choosing $N$ large enough.

Therefore as $\beta>d(1/p-1/2)$, combining (\ref{Estimate: g_{j,>l} for S_{21}}), (\ref{Estimate: g_{j, <l} for S_{21} main part}) and (\ref{Estimate: g_{j,<l} for S_{21} error part}), we get
\begin{align*}
    S_{21} &\leq C 2^{-l\varepsilon} \|b\|_{BMO^{\varrho}(\mathbb{R}^d)} \|F\|_{W^2 _{\beta}}\, \|f \|_{L^q}.
\end{align*}

\underline{\emph{Estimation of $S_{22}$}:} 
\underline{Case I:} Assume $|x_j'| \geq 2^{l+2} t$. 

Our argument is similar to Case I of $S_{12}$. Therefore using H\"older's inequality, \Cref{Theorem: Restriction Estimate}, (\ref{Inequality: Calculation with cutoff phi}), H\"older's inequality and (\ref{Inequality: BMO norm calculation}) we have
\begin{align*}
     & \|P_{4 B_j}\,  ((1-\psi) F^{(l)})(t\sqrt{\mathcal{L}}) P_{\widetilde{B}_j}(b - b_{B_j}) f\|_{L^q} \\
     & \leq |B(x_j, 2^{l+2} t)|^{1/q - 1/2} \| P_{4 B_j}\, ((1-\psi) F^{(l)})(t\sqrt{\mathcal{L}}) P_{\widetilde{B}_j}(b - b_{B_j}) f \|_{L^2} \\
     &\leq C(2^{l} t)^{(d_1+ d_2)(1/q - 1/2)} |x_j'|^{d_2(1/q - 1/2)} t^{ -(d_1 + d_2) (1/p - 1/2)}\, |x_j'|^{-d_2 (1/p -1/2)}\\
     & \hspace{3cm} 2^{-Nl} \|  F \|_{L^2} (2^{l} t)^{(d_1+d_2)(1/p - 1/q)} |x_j'|^{d_2(1/p - 1/q)}\,\|b\|_{BMO^{\varrho}(\mathbb{R}^d)} \|P_{\widetilde{B}_j} f\|_{L^q}\\
     &\leq C 2^{ld(1/p -1/2)}\, 2^{-Nl}\, \, \|  F \|_{L^2}\, \|b\|_{BMO^{\varrho}(\mathbb{R}^d)} \|P_{\widetilde{B}_j} f\|_{L^q} .
\end{align*}
So that choosing $N$ large enough we get
\begin{align*}
    S_{22} &\leq C 2^{-l\varepsilon} \|b\|_{BMO^{\varrho}(\mathbb{R}^d)} \|  F \|_{L^2}\, \| f\|_{L^q} .
\end{align*}

\underline{Case II:} Assume $|x_j'| < 2^{l+2} t$.

As in Case II of $S_{12}$ we write
\begin{align*}
    & P_{4 B_j}\, ((1-\psi) F^{(l)})(t\sqrt{\mathcal{L}}) P_{\widetilde{B}_j} (b - b_{B_j})\, f \\
    &= P_{4 B_j}\,  \Bigl(\sum_{\iota \geq 0} + \sum_{\iota \leq -6}\Bigr)  \, \left( \sum_{M=0 } ^l + \sum_{M=l+ 1 } ^{\infty}\right) F^{(l,\iota)}_M (\mathcal{L}, T) P_{\widetilde{B}_j} (b - b_{B_j})f \\
    &=:\Bigl(\sum_{\iota \geq 0} + \sum_{\iota \leq -6}\Bigr) \Tilde{g}_{j, \leq l} ^{\iota} + \Bigl(\sum_{\iota \geq 0} + \sum_{\iota \leq -6}\Bigr) \Tilde{g}_{j, > l} ^{\iota}.
\end{align*}

As (\ref{Estimate of S_{12} case II second term}) and (\ref{Estimate: g_{j,>l} for S_{12}}) we can estimate the second term. Indeed, using H\"older's inequality, \Cref{Theorem: Restriction Estimate}, (\ref{Inequality: Calculation with cutoff phi}), and again H\"older's inequality along with (\ref{Inequality: BMO norm calculation}), we get
\begin{align}
\label{Estimate: g_{j,>l} for S_{22}}
   & \left\|\Bigl(\sum_{\iota \geq 0} + \sum_{\iota \leq -6}\Bigr) \Tilde{g}_{j, > l} ^{\iota} \right\|_{L^q}\\
   & \nonumber \leq C(2^l t)^{Q(1/q - 1/2)}\,   \Bigl(\sum_{\iota \geq 0} + \sum_{\iota \leq -6}\Bigr)(2^{\iota} t^{-1})^{Q(1/p - 1/2)}\, 2^{-ld_2 (1/p - 1/2)}\,  2^{-N(l + \max{\{\iota,0}\})}\\
   & \nonumber \hspace{7cm} \| F \|_{L^2} \,(2^l t)^{Q(1/p - 1/q)}\,\|b\|_{BMO^{\varrho}(\mathbb{R}^d)} \|P_{\widetilde{B}_j} f\|_{L^q}\\
   & \nonumber \leq C  2 ^{ld(1/p - 1/2)}\, 2^{-Nl} \|F\|_{L^2} \, \|b\|_{BMO^{\varrho}(\mathbb{R}^d)} \|P_{\widetilde{B}_j} f\|_{L^q} \\
   & \nonumber \leq C 2^{-l \varepsilon} \|F\|_{L^2} \, \|b\|_{BMO^{\varrho}(\mathbb{R}^d)} \|P_{\widetilde{B}_j} f\|_{L^q} ,
\end{align}
by choosing $N$ sufficiently large.

For the first term $  \Bigl(\sum_{\iota \geq 0} + \sum_{\iota \leq -6}\Bigr) \Tilde{g}_{j, \leq l} ^{\iota}$, with the same notation as in Case I of $S_{11}$, let us set $\Tilde{g}_{j, m} ^{\iota,M} = P_{4 B_j} \, F^{(l,\iota)}_M (\mathcal{L}, T) P_{\widetilde{B}_{j, m} ^M} (b - b_{B_j}) f $ and write 
\begin{align*}
   \Bigl(\sum_{\iota \geq 0} + \sum_{\iota \leq -6}\Bigr) \Tilde{g}_{j, \leq l} ^{\iota} = \Bigl(\sum_{\iota \geq 0} + \sum_{\iota \leq -6}\Bigr) \sum_{M= 0} ^l \sum_{m= 1} ^{N_M} \chi_{\widetilde{B_{j, m} ^M}} \Tilde{g}_{j, m} ^{\iota,M} + \Bigl(\sum_{\iota \geq 0} + \sum_{\iota \leq -6}\Bigr) \sum_{M= 0} ^l \sum_{m= 1} ^{N_M}(1 - \chi_{\widetilde{B_{j, m} ^M}}) \Tilde{g}_{j, m} ^{\iota,M} .
\end{align*}

Choose $r$ such that $p < r < q < 2$. Then using H\"older's inequality, \Cref{Theorem: Restriction Estimate}, (\ref{Inequality: Calculation with cutoff phi}) and again H\"older's inequality with (\ref{Inequality: BMO norm calculation}),  we have
\begin{align*}
    & \left\| \Bigl(\sum_{i \geq 0} + \sum_{i \leq -6}\Bigr) \chi_{\widetilde{B_{j, m} ^M}} \Tilde{g}_{j, m} ^{\iota,M} \right\|_{L^q}\\
    & \leq C (2^l t)^{d_1(1/q - 1/2)}\, (2^{\gamma l} 2^M t \cdot 2^l t)^{d_2(1/q - 1/2)} \left\| \Bigl(\sum_{\iota \geq 0} + \sum_{\iota \leq -6}\Bigr) F^{(l,\iota)}_M (\mathcal{L}, T) P_{\widetilde{B}_{j, m} ^M}(b - b_{B_j}) f \right\|_{L^2}\\
    & \leq C  (2^l t)^{d_1(1/q - 1/2)}\, (2^{\gamma l} 2^M t \cdot 2^l t)^{d_2(1/q - 1/2)} \Bigl(\sum_{\iota \geq 0} + \sum_{\iota \leq -6}\Bigr) (2^{\iota} t^{-1})^{Q(1/p - 1/2)} 2^{-Md_2(1/p - 1/2)}\,  \\
    & \hspace{6cm} \|\phi_{\iota} F^{(l)}\|_{L^{\infty}}\, \|P_{4B_j} (b - b_{B_j}) \|_{L^{\frac{rp}{r-p}}}  \, \|P_{\widetilde{B}_{j, m} ^M} f\|_{L^r}  \\
    & \leq C  (2^l t)^{d_1(1/q - 1/2)}\, (2^{\gamma l} 2^M t \cdot 2^l t)^{d_2(1/q - 1/2)} \Bigl(\sum_{\iota \geq 0} + \sum_{\iota \leq -6}\Bigr) (2^{\iota} t^{-1})^{Q(1/p - 1/2)} 2^{-Md_2(1/p - 1/2)}\,  \\
    & 2^{-N(l + \max{\{\iota ,0}\})} \|  F \|_{L^2}\, (2^l t )^{Q(1/p -1/r)}\, \|b\|_{BMO^{\varrho}(\mathbb{R}^d)} (2^l t)^{d_1(1/r - 1/q)}\, (2^M t\cdot
    2^l t)^{d_2(1/r - 1/q)} \,\|P_{\widetilde{B}_{j, m} ^M} f\|_{L^q}\\
    &\leq C 2^{l\varepsilon} 2^{ld(1/r-1/2)} 2^{Q(1/p-1/r)} 2^{-Md_2(1/p-1/r)} 2^{-lN} \|F\|_{L^2}\, \|b\|_{BMO^{\varrho}(\mathbb{R}^d)}\, \|P_{\widetilde{B}_{j, m} ^M} f\|_{L^q} \\
    &\leq C 2^{l\varepsilon} 2^{ld(1/r-1/2)} 2^{-lN} \|F\|_{L^2}\, \|b\|_{BMO^{\varrho}(\mathbb{R}^d)}\, \|P_{\widetilde{B}_{j, m} ^M} f\|_{L^q} ,
\end{align*}
for some $\varepsilon>0$, by choosing $r$ very close to $p$ and $\gamma>0$ small enough before.
 
Hence as in (\ref{Estimate: g_{j, <l} for S_{12} main part}) we get
\begin{align}
\label{Estimate: g_{j, <l} for S_{22} main part}
 \left\| \Bigl(\sum_{\iota \geq 0} + \sum_{\iota \leq -6}\Bigr)  \sum_{M= 0} ^l \sum_{m= 1} ^{N_M} \chi_{\widetilde{B_{j, m} ^M}} g_{j, m} ^{\iota, M} \right\|_{L^q} ^q  & \leq C 2^{-lq\varepsilon} \,  \|b\|_{BMO^{\varrho}(\mathbb{R}^d)} ^q\, \|F\|_{L^2} ^q\,\|P_{\widetilde{B}_j} f\|_{L^q} ^q ,
\end{align}
by choosing $N$ large enough.

On the other hand, a similar argument as in the proof of (\ref{interpolation in remaindeder term for S_{12}}), shows that for all $1\leq p < 2$,
\begin{align*}
   &  \left \| \sum_{M= 0} ^l \sum_{m= 1} ^{N_M}(1 - \chi_{\widetilde{B_{j, m} ^M}})  F^{(l,\iota)}_M (\mathcal{L}, T) P_{\widetilde{B}_{j, m} ^M} (b- b_{B_j}) f \right \|_{L^2}\\
    &\leq C 2^{l\varepsilon }\, t^{-Q(1/p - 1/2)}\, 2^{\iota Q (1/p-1/2)} \, 2^{l d_2 (1-1/p)}\, \|\phi_{\iota} F^{(l)}\|_{L^{\infty}}\, \|P_{\widetilde{B}_{j} } (b- b_{B_j})f \|_{L^p} .
\end{align*}

The above estimate together with H\"older's inequality and (\ref{Inequality: Calculation with cutoff phi}) and again H\"older's inequality with (\ref{Inequality: BMO norm calculation}),  implies that
\begin{align}
\label{Estimate: g_{j,<l} for S_{22} error part}
    & \left \| P_{4B_j} \sum_{M=0} ^l \sum_{m= 1} ^{N_M} ( 1- \chi_{\widetilde{B_{j, m} ^M}}) F^{(l,\iota)} _M (\mathcal{L}, T) P_{\widetilde{B}_{j, m} ^M} (b - b_{B_j} )f  \right \|_{L^q}\\
    & \nonumber \leq (2^l t)^{Q(1/q - 1/2)}\, \Bigl(\sum_{\iota \geq 0} + \sum_{\iota \leq -6}\Bigr) \,2^{l\varepsilon }\, t^{-Q(1/p - 1/2)}\, 2^{\iota Q (1/p-1/2)} \, 2^{l d_2 (1-1/p)}\, \\
    & \nonumber\hspace{7cm} \|\phi_{\iota} F^{(l)}\|_{L^{\infty}}\, \|P_{\widetilde{B}_{j} }(b - b_{B_j}) f \|_{L^p}\\
   & \nonumber \leq C (2^l t)^{Q(1/q - 1/2)}\, \Bigl(\sum_{\iota \geq 0} + \sum_{\iota \leq -6}\Bigr)  2^{l\varepsilon }\,  t^{-Q(1/p - 1/2)} 2^{\iota Q (1/p-1/2)}  \, 2^{l d_2 (1-1/p)}\, 2^{-N(l + \max{\{\iota,0}\})} \,  \|F\|_{L^2}\, \\
   & \nonumber \hspace{8cm} \|P_{\widetilde{B}_j}(b - b_{B_j})\|_{L^{\frac{pq}{q-p}}} \|P_{\widetilde{B}_j} f\|_{L^q} \\
   & \nonumber \leq C (2^l t)^{Q(1/q - 1/2)}\,  2^{l\varepsilon }\,  t^{-Q(1/p - 1/2)}\,  2^{-Nl } \, 2^{l d_2 (1-1/p)}\, \|F\|_{L^2}\, \\
   & \nonumber \hspace{8cm} (2^l t) ^{Q(1/p - 1/q)}\,
   \|b\|_{BMO^{\varrho}(\mathbb{R}^d)}  \|P_{\widetilde{B}_j} f\|_{L^q} \\
   & \nonumber \leq C  2^{-l\varepsilon }\, \|F\|_{L^2}\, \|b\|_{BMO^{\varrho}(\mathbb{R}^d)}\, \|P_{\widetilde{B}_{j} } f \|_{L^q} ,
\end{align}
by choosing $N$ large enough.

Summarizing (\ref{Estimate: g_{j,>l} for S_{22}}), (\ref{Estimate: g_{j, <l} for S_{22} main part}) and (\ref{Estimate: g_{j,<l} for S_{22} error part}), we have
\begin{align*}
    S_{22} &\leq C 2^{-l\varepsilon} \|b\|_{BMO^{\varrho}(\mathbb{R}^d)} \|F\|_{L^2}\, \|f \|_{L^q}.
\end{align*}

$(4)$ \emph{Estimate of the commutator:}
 Aggregating all the estimates of $S_1$ and $S_2$, we get 
\begin{align*}
    \|[b, F^{(l)}(t\sqrt{\mathcal{L}})f]\|_{L^q} &\leq S_1^{1/q} + S_2^{1/q}\nonumber \\
    & \leq C 2^{-l \varepsilon} \|b\|_{BMO^{\varrho}(\mathbb{R}^d)} \|F\|_{W^2_{\beta}} \|f\|_{L^q} + C 2^{-l\varepsilon} \|b\|_{BMO^{\varrho}(\mathbb{R}^d)} \|  F \|_{L^2}\, \| f\|_{L^q}\nonumber \\
    & \leq C 2^{-l \varepsilon} \|b\|_{BMO^{\varrho}(\mathbb{R}^d)} \|F\|_{W^2_{\beta}} \|f\|_{L^q}\nonumber .
\end{align*}
Finally, summing over $l$, we conclude that for all $p< q<2$ and $\beta>d(1/p-1/2)$, 
\begin{align*}
    \|[b, F(t\sqrt{\mathcal{L}})f]\|_{L^q} & \leq C \sum_{l\geq 0} 2^{-l \varepsilon} \|b\|_{BMO^{\varrho}(\mathbb{R}^d)} \|F\|_{W^2_{\beta}} \|f\|_{L^q} \\
    & \leq C \|b\|_{BMO^{\varrho}(\mathbb{R}^d)} \|F\|_{W^2_{\beta}} \|f\|_{L^q},
\end{align*}
which completes the proof of the \Cref{Theorem: Multiplier for Commutator}. 
    
\end{proof}

\section{Compactness of Bochner-Riesz Commutator}

\begin{proof}[Proof of \Cref{Theorem: Compactness of Bochner-Riesz commutator}]
Let us set $F(\eta) = (1-\eta^2)_{+}^{\alpha}$. Recall that from (\ref{Equation: Dyadic decomposition}) we get $[b, F(\sqrt{\mathcal{L}})] = \sum_{l \geq 0} [b, F^{(l)} (\sqrt{\mathcal{L}})]$. Using (\ref{equation: inequality for F^l}) one can easily check that $\sum_{l =0} ^N [b, F^{(l)}(\sqrt{\mathcal{L}})]$ converges to $[b, F(\sqrt{\mathcal{L}})]$ in $L^q$-norm as $N \rightarrow \infty$. Therefore it is enough to show that for each $l\geq 0$ the operator $[b, F^{(l)}(\sqrt{\mathcal{L}})]$ is compact. Since $C_c ^{\infty} (\mathbb{R}^d)$ is dense in $CMO(\mathbb{R}^d)$, it suffices to show that for any arbitrary bounded set $\mathcal{F}$ in $L^q(\mathbb{R}^d)$ and $b \in C_c ^{\infty} (\mathbb{R}^d)$, the set $\mathcal{E}= \{ [b, F^{(l)}(\sqrt{\mathcal{L}})]f: f \in \mathcal{F}\}$ satisfies all the conditions (1), (2) and (3) of \Cref{Theorem: Characterizations of compactness}.

The condition (1) for $\mathcal{E}$ is straightforward. Indeed, from (\ref{equation: inequality for F^l}) we have  
\begin{align*}
   \sup_{f \in \mathcal{F}} \|[b, F^{(l)}(\sqrt{\mathcal{L}})]f \|_{L^q (\mathbb{R}^d)}
    & \leq C \, \sup_{f \in \mathcal{F}} \|b\|_{BMO} \|F\|_{W^2 _{\beta}} \|f\|_{L^q (\mathbb{R}^d)} < \infty.
\end{align*}

Before proceeding to the proof of the conditions (2) and (3), let us choose a function $\psi\in C_c ^{\infty} (-4,4)$ such that $\psi(\lambda) = 1$ on $(-2,2)$. Note that $1-\psi$ is supported outside $(-2,2)$, so we can choose $\phi \in C_c^{\infty}(2,8)$ such that $1-\psi(\lambda) = \sum_{\iota \geq 0} \phi(2^{-\iota }\lambda)$ for all $\lambda>0$. Then we may write $F^{(l)}(\lambda) = (\psi F^{(l)})(\lambda) + \sum_{\iota \geq 0} (\phi_{\iota} F^{(l)})(\lambda)$ where $\phi_{\iota}(\lambda) = \phi(2^{-\iota }\lambda)$ for $\iota>0$.

Using Sobolev embedding theorem one can easily check that
\begin{align}
\label{Calculation: F^{l} with psi}
    \|\psi F^{(l)}\|_{L^{\infty}_{N}} \leq C 2^{l(N+1/2+\varepsilon)} \| F \|_{L^2},
\end{align}
 and as in (\ref{Inequality: Calculation with cutoff phi}), for any $M>0$,
\begin{align}
\label{Calculation: F^{l} with phi}
     \|\phi_{\iota} F^{(l)}\|_{L^{\infty}_{N}} \leq C 2^{-M(l + \iota)} \| F \|_{L^2} .
\end{align}
As $C_c^{\infty}(\mathbb{R}^d)$ is dense in $L^q(\mathbb{R}^d)$, for $f \in \mathcal{F}$ and any $\varepsilon>0$ there exists $g \in C_c^{\infty}(\mathbb{R}^d)$ such that $\|f-g\|_{L^q} < \varepsilon$. 

Now let us first check the condition (3) for $\mathcal{E}$. Note that
\begin{align*}
    \|\chi_{\mathcal{F}_A} [b, F^{(l)}(\sqrt{\mathcal{L}})]f \|_{L^q} &\leq \|\chi_{\mathcal{F}_A} [b, F^{(l)}(\sqrt{\mathcal{L}})](f-g) \|_{L^q} + \|\chi_{\mathcal{F}_A} [b, F^{(l)}(\sqrt{\mathcal{L}})]g \|_{L^q} \\
    &\leq \varepsilon + \|\chi_{\mathcal{F}_A} [b, F^{(l)}(\sqrt{\mathcal{L}})]g \|_{L^q} .
\end{align*}
Therefore enough to show that $\|\chi_{\mathcal{F}_A} [b, F^{(l)}(\sqrt{\mathcal{L}})]g \|_{L^q} < \varepsilon$ for $g \in C_c^{\infty}(\mathbb{R}^d)$. As $b \in C_c ^{\infty} (\mathbb{R}^d)$, we may choose $j\in \mathbb{N}$ such that $\supp{b} \subseteq B(0, 2^j)$. Then for $\varrho (x, 0) > A = 2\cdot 2^{2j}$, we estimate the following.
\begin{align*}
     \|\chi_{\mathcal{F}_A} [b, F^{(l)}(\sqrt{\mathcal{L}})]g \|_{L^q} &\leq \|\chi_{\mathcal{F}_A} [b, (\psi F^{(l)})(\sqrt{\mathcal{L}})]g \|_{L^q} + \sum_{\iota \geq 0} \|\chi_{\mathcal{F}_A} [b, (\phi_{\iota} F^{(l)})(\sqrt{\mathcal{L}})]g \|_{L^q} .
 \end{align*}
Let us first estimate the second term. Note that
 \begin{align*}
     \|\chi_{\mathcal{F}_A} [b, (\phi_{\iota} F^{(l)})(\sqrt{\mathcal{L}})]g \|_{L^q (\mathbb{R}^d)} & \leq \Biggl\{  \int_{\varrho(x, 0) \geq 2\cdot 2^{2j}} \Bigg | \int_{B(0, 2^j)} K_{(\phi_{\iota} F^{(l)})(\sqrt{\mathcal{L}})} (x,y)\,  b(y)\, g(y)\, dy   \Bigg|^q dx \Biggl\}^{1/q}.
 \end{align*}
Now, for $y \in B(0, 2^j)$, using \Cref{Theorem: Weighted Plancherel with L infinity condition}, (\ref{Calculation: F^{l} with phi}) and \Cref{lemma: outside distance} we have 
\begin{align*}
   & \int_{\varrho(x, 0) > 2 \cdot 2^{2j}} | K_{ (\phi_{\iota} F^{(l)})(\sqrt{\mathcal{L}})} (x,y)|\, dx \\
   & \leq C \int_{\varrho(x, y) >  2^{2j}} \frac{(1+2^{\iota} \varrho(x,y))^N |K_{ (\phi_{\iota} F^{(l)})(\sqrt{\mathcal{L}})} (x,y)| }{(1+ 2^{\iota} \varrho(x,y))^N }\, dx\\
   &\leq C |B(y, 2^{-\iota})|^{-1}\, \|\phi_{\iota} F^{(l)}\|_{L^{\infty} _{N+1/2+ \varepsilon}}\, \int_{\varrho(x, y) > 2^{2j} } \frac{dx}{(1+ 2^{\iota} \varrho(x,y))^N}\\
   & \leq C |B(y, 2^{-\iota})|^{-1} 2^{-M(l+\iota)} \, 2^{-\iota N} 2^{2j(-N + d_1 + d_2)}\, \max\{2^{2j}, |y'|\}^{d_2}\\
   & \leq C 2^{\iota Q} 2^{-\iota M} 2^{j(-2N + 2d_1 + 4d_2)} ,
\end{align*}
for all $M> 0$ and $N>Q$.

On the other hand, using \Cref{Theorem: Weighted Plancherel with L infinity condition} and (\ref{Calculation: F^{l} with phi}) for any $M>0$,
\begin{align*}
    \int_{B(0, 2^{j})} |    K_{ (\phi_{\iota} F^{(l)})(\sqrt{\mathcal{L}})} (x,y)|\, dy & \leq C |B(0, 2^{j}) | |B(x, 2^{-\iota})|^{-1} \| \phi_{\iota} F^{(l)} \|_{L^{\infty}_{1/2+\varepsilon}} \\
    & \leq C 2^{jQ} 2^{\iota Q} 2^{-M(l+\iota)} \leq C 2^{jQ} 2^{\iota Q} 2^{-\iota M}.
\end{align*}
Then choosing $N$ sufficiently large and $M> Q$, by Schur's lemma(see \cite{Grafakos_Modern_Fourier_Analysis_2014}) for $1<q<\infty$, as $j \to \infty$ we get 
\begin{align}
\label{Inequality: Condition 3 for phi F}
    \sum_{\iota \geq 0} \|\chi_{\mathcal{F}_A} [b, (\phi_{\iota} F^{(l)})(\sqrt{\mathcal{L}})]g \|_{L^q (\mathbb{R}^d)} \leq C \|g\|_{L^q(\mathbb{R}^d)} 2^{jQ(1-\frac{1}{q})} 2^{j(-2N + 2d_1 + 4d_2)/q} \sum_{\iota \geq 0} 2^{-\iota (M-Q)} < \varepsilon .
\end{align}
Estimate for the first term is also similar. Using \Cref{Theorem: Weighted Plancherel with L infinity condition}, (\ref{Calculation: F^{l} with psi}), \Cref{lemma: outside distance} and taking $j>l$, we have
\begin{align*}
   \int_{\varrho(x, 0) > 2 \cdot 2^{2j}} | K_{ (\psi F^{(l)})(\sqrt{\mathcal{L}})} (x,y)|\, dx & \leq C \int_{\varrho(x, y) >  2^{2j}} \frac{(1+ \varrho(x,y))^N |K_{ (\psi F^{(l)})(\sqrt{\mathcal{L}})} (x,y)| }{(1+ \varrho(x,y))^N }\, dx\\
   &\leq C |B(y, 1)|^{-1}\, \|\psi F^{(l)}\|_{L^{\infty}_{N+1/2+ \varepsilon}}\, \int_{\varrho(x, y) > 2^{2j} } \frac{dx}{(1+ \varrho(x,y))^N}\\
   & \leq C |B(y, 1)|^{-1} 2^{l (N+1+ 2\varepsilon)}\, 2^{2j(-N + d_1 + d_2)}\, \max\{2^{2j}, |y'|\}^{d_2}\\
   & \leq C 2^{j(-N + 2d_1 + 4d_2+1+2\varepsilon)} .
\end{align*}
And using \Cref{Theorem: Weighted Plancherel with L infinity condition} and (\ref{Calculation: F^{l} with psi}) for $j>l$, we have 
\begin{align*}
    \int_{B(0, 2^{j})} | K_{ (\psi F^{(l)})(\sqrt{\mathcal{L}})} (x,y)|\, dy & \leq C |B(0, 2^{j}) | |B(x, 1)|^{-1} \| \psi F^{(l)} \|_{L^{\infty}_{1/2+\varepsilon}} \\
    & \leq C 2^{jQ} 2^{l(1+2\varepsilon)} \leq C 2^{j(Q+1+2\varepsilon)}.
\end{align*}
Therefor choosing $N$ sufficiently large, by Schur's lemma for $1<q<\infty$, as $j \to \infty$ we get
\begin{align}
\label{Inequality: Condition 3 for psi F}
    \|\chi_{\mathcal{F}_A} [b, (\psi F^{(l)})(\sqrt{\mathcal{L}})]g \|_{L^q(\mathbb{R}^d)} &\leq C \|g\|_{L^q(\mathbb{R}^d)} 2^{j(1+2\varepsilon)} 2^{jQ(1-\frac{1}{q})} 2^{j(-N + 2d_1 + 4d_2)/q} <\varepsilon .
\end{align}
Then combining estimates (\ref{Inequality: Condition 3 for phi F}) and (\ref{Inequality: Condition 3 for psi F}) as $A \to \infty$ we have 
\begin{align*}
    \|\chi_{\mathcal{F}_A} [b, F^{(l)}(\sqrt{\mathcal{L}})]g \|_{L^q(\mathbb{R}^d)} < \varepsilon.
\end{align*}
Therefore it remains to check only the condition (2) for $\mathcal{E}$. Again note that
\begin{align*}
    & \|[b, F^{(l)}(\sqrt{\mathcal{L}})] f(\cdot + t) - [b, F^{(l)}(\sqrt{\mathcal{L}})]f\|_{L^q} \\
    &\leq \|[b, F^{(l)}(\sqrt{\mathcal{L}})] (f-g)(\cdot + t) - [b, F^{(l)}(\sqrt{\mathcal{L}})](f-g)\|_{L^q} \\
    & \hspace{6cm} + \|[b, F^{(l)}(\sqrt{\mathcal{L}})] g(\cdot + t) - [b, F^{(l)}(\sqrt{\mathcal{L}})]g \|_{L^q} \\
    &\leq 2\|[b, F^{(l)}(\sqrt{\mathcal{L}})] (f-g)\|_{L^q} + \|[b, F^{(l)}(\sqrt{\mathcal{L}})] g(\cdot + t) - [b, F^{(l)}(\sqrt{\mathcal{L}})]g \|_{L^q} \\
    &\leq \varepsilon + \|[b, F^{(l)}(\sqrt{\mathcal{L}})] g(\cdot + t) - [b, F^{(l)}(\sqrt{\mathcal{L}})]g \|_{L^q} .
\end{align*}
Therefore enough to verify condition (2) for $g \in C_c^{\infty}(\mathbb{R}^d)$. We can write 
\begin{align*}
  & [b, F^{(l)}(\sqrt{\mathcal{L}})] g(x + t) - [b, F^{(l)}(\sqrt{\mathcal{L}})]g (x) \\
  & = \int_{\varrho(x,y) > 2 \varepsilon^{-1} \varrho(x+t, x)} \{b(x+t) - b(x)\} K_{F^{(l)}(\sqrt{\mathcal{L}})} (x, y) g(y)\, dy\\
  & + \int_{\varrho(x,y) > 2 \varepsilon^{-1} \varrho(x+t, x)} \{b(y) - b(x+t)\}\{ K_{F^{(l)}(\sqrt{\mathcal{L}})} (x, y) - K_{F^{(l)}(\sqrt{\mathcal{L}})} (x+ t, y)\}  g(y)\, dy\\
  & +  \int_{\varrho(x,y) \leq 2 \varepsilon^{-1} \varrho(x+t, x)} \{b(y) - b(x)\} K_{F^{(l)}(\sqrt{\mathcal{L}})} (x, y) g(y)\, dy\\
  & -  \int_{\varrho(x,y) \leq 2 \varepsilon^{-1} \varrho(x+t, x)} \{b(y) - b(x+t)\} K_{F^{(l)}(\sqrt{\mathcal{L}})} (x+t, y) g(y)\, dy\\
  &=: J_1 + J_2 + J_3 + J_4.
\end{align*}
In the following, we estimate each $J_i$ separately for $i=1,2,3,4$.

\underline{\emph{Estimate of $J_1$}}:
Since $b \in C_c ^{\infty} (\mathbb{R}^d)$, by \Cref{Theorem: Mean value theorem} we have $| b(x + t) - b(x) | \leq C \varrho(x+t, x)$. Then we can write
\begin{align*}
    |J_1| &\leq C \varrho(x+t, x) \left[\int_{\varrho(x,y) > 2 \varepsilon^{-1} \varrho(x+t, x)} |K_{(\psi F^{(l)})(\sqrt{\mathcal{L}})} (x, y)| |g(y)|\, dy \right. \\
    & \hspace{6cm} + \left. \sum_{\iota \geq 0} \int_{\varrho(x,y) > 2 \varepsilon^{-1} \varrho(x+t, x)} |K_{(\phi_{\iota} F^{(l)})(\sqrt{\mathcal{L}})} (x, y)| |g(y)|\, dy \right] \\
    &=:J_{11} + J_{12} .
\end{align*}
For $0<\tau<1$, using \Cref{Theorem: Weighted Plancherel with L infinity condition} and (\ref{Calculation: F^{l} with phi}),
\begin{align*}
    |J_{12}| &\leq C \varrho(x+t, x) \sum_{\iota \geq 0} \int_{\varrho(x,y) > 2 \varepsilon^{-1} \varrho(x+t, x)} |K_{(\phi_{\iota} F^{(l)})(\sqrt{\mathcal{L}})} (x, y)| |g(y)|\, dy \\
    &\leq C \varrho(x+t, x) \sum_{\iota \geq 0} \sum_{k=1}^{\infty} \int_{2^{k} \varepsilon^{-1} \varrho(x+t, x)< \varrho(x, y) \leq 2^{k+1} \varepsilon^{-1} \varrho(x+t, x)} \\
    & \hspace{5cm} \frac{(1+2^{\iota} \varrho(x,y))^{Q+\tau} |K_{(\phi_{\iota} F^{(l)})(\sqrt{\mathcal{L}})} (x, y)|}{(1+2^{\iota}\varrho(x,y))^{Q+\tau}} |g(y)|\, dy \\
    &\leq C \varrho(x+t, x) \sum_{\iota \geq 0} \sum_{k=1} ^{\infty} \frac{|B(x,2^{-\iota})|^{-1} \|\phi_{\iota} F^{(l)}\|_{L^{\infty}_{Q + \tau +1/2+ \varepsilon} } |B(x, 2^{k+1} \varepsilon^{-1} \varrho(x+t, x))|  }  {\{2^{\iota} 2^k \varepsilon^{-1} \varrho(x+t, x)\}^{Q + \tau}} \times \\
    & \hspace{5cm} \frac{1}{|B(x, 2^{k+1} \varepsilon^{-1} \varrho(x+t, x))| }\, \int_{\varrho(x, y) \leq 2^{k+ 1} \varepsilon^{-1} \varrho(x+t, x) } |g(y)|\, dy\\
    & \leq C \varrho(x+t, x) \sum_{\iota \geq 0} \sum_{k=1} ^{\infty} \frac{2^{\iota Q} 2^{-M(l+\iota)} \bigl( 2^{k+1} \varepsilon^{-1} \varrho(x+t, x) \bigr)^Q }  {\{ 2^{\iota} 2^k \varepsilon^{-1} \varrho(x+t, x)\}^{Q + \tau}}  \mathcal{M}g(x)   \\
    & \leq  C \varepsilon^{\tau} \varrho(x+t, x)^{1-\tau} \mathcal{M}g(x) \left(\sum_{k=1} ^{\infty} \frac{1}{2^{k\tau}} \right) \left(\sum_{\iota \geq 0} \frac{1}{2^{\iota(M+\tau)}}  \right) \\
    & \leq C \varepsilon^{\tau} \varrho(x+t, x)^{1-\tau} \mathcal{M}g(x),
\end{align*}
by choosing $M>0$.

Estimate of $J_{11}$ is similar. Using \Cref{Theorem: Weighted Plancherel with L infinity condition} and (\ref{Calculation: F^{l} with psi}) for $0<\tau<1$ we have $|J_{11}| \leq C \varepsilon^{\tau} 2^{l(Q +\tau + 1 + 2\varepsilon)} \varrho(x+t, x)^{1-\tau} \mathcal{M}g(x)$. Therefore combining both estimates for fixed $l$, as $t \to 0$ we get
\begin{align*}
    \|J_1\|_{L^q(\mathbb{R}^d)} &\leq C\varepsilon \|\mathcal{M}g\|_{L^q(\mathbb{R}^d)} \leq C \varepsilon \, \|g\|_{L^q(\mathbb{R}^d)} \leq C \varepsilon.
\end{align*}

\underline{\emph{Estimate of $J_2$}}:
Note that $|b(y) - b(x+t)| \leq 2 \|b\|_{L^{\infty}} \leq C$. As in $J_1$ here we also break $J_2$ into two parts $J_{21}$ and $J_{22}$ corresponding to $\psi$ and $\sum_{\iota \geq 0} \phi_{\iota}$. Then using \Cref{Theorem: Mean value theorem}, \Cref{Theorem: Weighted Plancherel with L infinity condition}, (\ref{Calculation: F^{l} with phi}), for $0 < \tau < 1$, as estimated in $J_{12}$, for $J_{22}$ with $M>1$ we get 
\begin{align*}
    &|J_{22}| \leq C \sum_{\iota \geq 0} \int_{\varrho(x, y) > 2 \varepsilon^{-1} \varrho(x+t, x)} \Big| K_{(\phi_{\iota} F^{(l)})(\sqrt{\mathcal{L}})} (x, y) -K_{(\phi_{\iota} F^{(l)})(\sqrt{\mathcal{L}})} (x+t, y) \Big| \, \big|g(y)\big| \, dy\\
    & \leq C \varrho(x+t, x) \sum_{\iota \geq 0} \sum_{k=1} ^{\infty} \int_{2^{k} \varepsilon^{-1} \varrho(x+t, x)< \varrho(x, y) \leq 2^{k+1} \varepsilon^{-1} \varrho(x+t, x)} \int_{0}^{1} \\
    &  \hspace{4cm} \frac{(1+2^{\iota} \varrho(\gamma_0 (s),y))^{Q+\tau} |X K_{(\phi_{\iota} F^{(l)})(\sqrt{\mathcal{L}})} (\gamma_{0}(s), y)|}{(1+2^{\iota} \varrho(\gamma_0 (s),y))^{Q+\tau}} |g(y)|\, dy \ ds \\
    & \leq C \varepsilon^{\tau} \varrho(x+t, x)^{1-\tau} \mathcal{M}g(x),
\end{align*}
where we have used the fact that for $t$ small, $\varrho(x,y) \sim \varrho(\gamma_0(s), y)$.

For $J_{21}$ similarly we get $|J_{21}| \leq C \varepsilon^{\tau} 2^{l(Q +\tau +1+ 2\varepsilon)} \varrho(x+t, x)^{1-\tau} \mathcal{M}g(x)$.
Therefore as $t \to 0$,
\begin{align*}
    \|J_2\|_{L^q(\mathbb{R}^d)} &\leq C \varepsilon \|\mathcal{M}g\|_{L^q(\mathbb{R}^d)} \leq C \varepsilon .
\end{align*}

\underline{\emph{Estimate of $J_3$}}:
Again as $J_1$ we write $J_3$ as sum of $J_{31}$ and $J_{32}$. Since $b \in C_c ^{\infty} (\mathbb{R}^d)$, we have $|b(y) - b(x)| \leq C \varrho(x, y)$. Using this fact, \Cref{Theorem: Weighted Plancherel with L infinity condition} and (\ref{Calculation: F^{l} with phi}) we see that 
\begin{align*}
   |J_{32}| &\leq C \sum_{\iota \geq 0} \int_{\varrho(x, y) \leq \varepsilon^{-1} \varrho(x+t, x)} \varrho(x, y) |K_{(\phi_{\iota} F^{(l)})(\sqrt{\mathcal{L}})} (x, y)| |g(y)|\, dy \\
   &\leq  C \varepsilon^{-1} \varrho(x+t,x) \sum_{\iota \geq 0} |B(x,2^{-\iota})|^{-1}\, \|\phi_{\iota} F^{(l)}\|_{L^{\infty}_{1/2+\varepsilon}}\, |B(x,2 \varepsilon^{-1} \varrho(x+t, x)) | \\
   & \hspace{5cm} \frac{1}{|B(x,2 \varepsilon^{-1} \varrho(x+t, x)) |} \int_{B(x, 2 \varepsilon^{-1} \varrho(x+t, x))}|g(y)| \, dy\\
   &\leq C \varepsilon^{-1} \varrho(x+t,x) \sum_{\iota \geq 0} 2^{\iota Q} 2^{-M(l+\iota)} (2 \varepsilon^{-1} \varrho(x+t, x))^Q\, \mathcal{M}g(x) \\
   & \leq C (\varepsilon^{-1} \varrho(x+t, x))^{Q+1} \, \mathcal{M}g(x) ,
\end{align*}
by choosing $M>Q$.

Similarly we also get $|J_{31}| \leq C 2^{l(1+2\varepsilon)} (\varepsilon^{-1} \varrho(x+t, x))^{Q+1} \, \mathcal{M}g(x)$. Therefore combining both the estimates as $t \to 0$ we get, 
\begin{align*}
   \|J_3\|_{L^q(\mathbb{R}^d)} \leq C \varepsilon \| \mathcal{M}g \|_{L^q(\mathbb{R}^d)} \leq C \varepsilon \|g \|_{L^q(\mathbb{R}^d)} \leq C \varepsilon.
\end{align*}

\underline{\emph{Estimate of $J_4$}}:
By \Cref{Theorem: Mean value theorem} and triangle inequality, we have 
 \begin{align*}
     &| b(x+ t) - b(y)| \leq  C \varrho (x+t, y) \leq C \{\varrho (x+t, x) + \varrho (x, y)\}.
 \end{align*}
Therefore similarly as in $J_3$ as $t \to 0$ we get
\begin{align*}
    \|J_4\|_{L^q(\mathbb{R}^d)}& \leq C 2^{l(1/2+\varepsilon)} (\varepsilon^{-1} \varrho(x+t, x))^{Q+1} \|\mathcal{M} g\|_{L^q(\mathbb{R}^d)} \leq C \varepsilon \|g\|_{L^q(\mathbb{R}^d)} \leq  C\varepsilon.
\end{align*}
Finally, gathering all the  estimates of $J_1, J_2, J_3$ and $J_4$ we conclude that 
\begin{align*}
   \lim_{t \rightarrow 0} \|[b, F^{(l)}(\sqrt{\mathcal{L}})] f(\cdot + t) - [b, F^{(l)}(\sqrt{\mathcal{L}})]f\|_{L^q(\mathbb{R}^d)} = 0 \,\,\, \mbox{uniformly}\,\,\, \mbox{in}\, \, \, f\in \mathcal{F}.
\end{align*}
This completes the proof of \Cref{Theorem: Compactness of Bochner-Riesz commutator}.
\end{proof}

\section*{Acknowledgements}
The first author was supported from his NBHM post-doctoral fellowship, DAE, Government of India. The second author would like to acknowledge the support of the Prime Minister's Research Fellows (PMRF) supported by Ministry of Education, Government of India. We would like to thank Sayan Bagchi for his careful reading of this article and several useful discussions.


\newcommand{\etalchar}[1]{$^{#1}$}
\providecommand{\bysame}{\leavevmode\hbox to3em{\hrulefill}\thinspace}
\providecommand{\MR}{\relax\ifhmode\unskip\space\fi MR }
\providecommand{\MRhref}[2]{%
  \href{http://www.ams.org/mathscinet-getitem?mr=#1}{#2}
}
\providecommand{\href}[2]{#2}

\end{document}